\documentclass{amsart}
\usepackage{amsfonts,amssymb,amscd,amsmath,latexsym,amsbsy,mathrsfs,bbm,textcomp,float}
\usepackage[all]{xy}
\pdfoutput=1
\usepackage{graphicx}

\numberwithin{equation}{section}
\theoremstyle{plain}
\newtheorem{thm}{Theorem}[section]

\newtheorem{prop}[thm]{Proposition}

\newtheorem{defi}[thm]{Definition}
\newtheorem{lem}[thm]{Lemma}
\newtheorem{cor}[thm]{Corollary}

\newtheorem{conv}[thm]{Convention}
\newtheorem{eg}[thm]{{Example}}

\theoremstyle{remark}
\newtheorem{rema}[thm]{Remark}

\newcommand{\ad}{{\mbox{\upshape{ad}}}}

\newcommand{\Aut}{\mathrm{Aut}}

\newcommand{\barU}{\overline{\phantom{u}}^U}
\newcommand{\barB}{\overline{\phantom{u}}^B}
\newcommand{\bc}{{\mathbf{c}}}

\newcommand{\bs}{{\mathbf{s}}}

\newcommand{\Bcs}{B_{\bc,\bs}}

\newcommand{\BK}{\mathrm{BK}}
\newcommand{\C}{{\mathbb C}}

\newcommand{\chU}{\check{U}}

\newcommand{\N}{{\mathbb N}}

\newcommand{\cA}{{\mathcal A}}

\newcommand{\cB}{{\mathcal B}}
\newcommand{\cC}{{\mathcal C}}

\newcommand{\cF}{{\mathcal F}}

\newcommand{\cJ}{{\mathcal J}}
\newcommand{\cK}{{\mathcal K}}

\newcommand{\cM}{{\mathcal M}}
\newcommand{\cO}{{\mathcal O}}

\newcommand{\cS}{{\mathcal S}}

\newcommand{\cV}{{\mathcal V}}

\newcommand{\coid}{B_{{\mathbf c},{\mathbf s}}}

\newcommand{\End}{\mbox{End}}

\newcommand{\field}{{\mathbb K}}
\newcommand{\For}{{\mathcal{F}or}}

\newcommand{\flip}{{\mathrm{flip}}}

\newcommand{\gfrak}{{\mathfrak g}}

\newcommand{\hfrak}{{\mathfrak h}}

\newcommand{\hght}{\mathrm{ht}}

\newcommand{\id}{{\mathrm{id}}}

\newcommand{\inv}{\mathrm{inv}}

\newcommand{\kfrak}{{\mathfrak k}}
\newcommand{\Kmat}{\mathscr{K}}

\newcommand{\kow}{{\varDelta}}
\newcommand{\lact}{{\triangleright}}

\newcommand{\ltil}{\tilde{l}}

\newcommand{\alphatil}{\tilde{\alpha}}

\newcommand{\Mcs}{\cM_{\bc,\bs}}

\newcommand{\lambdatil}{\tilde{\lambda}}

\newcommand{\lot}{\mbox{\upshape lot}}

\newcommand{\ophan}{\overline{\phantom{r}}}

\newcommand{\op}{\mathrm{op}}
\newcommand{\Oint}{\cO_{\mathrm{int}}}
\newcommand{\ot}{\otimes}

\newcommand{\rib}{v}
\newcommand{\roots}{\Pi}

\newcommand{\rank}{\mathrm{rank}}

\newcommand{\Rmat}{{\mathscr{R}}}

\newcommand{\scrC}{\mathscr{C}}
\newcommand{\scrK}{\mathscr{K}}

\newcommand{\slfrak}{{\mathfrak{sl}}}

\newcommand{\sB}{\mathscr{B}}
\newcommand{\sU}{{\mathscr{U}}}

\newcommand{\Uq}{U}

\newcommand{\uqg}{{U_q(\mathfrak{g})}}
\newcommand{\uq}{U}

\newcommand{\uqgX}{U_q(\mathfrak{g}_X)}

\newcommand{\UnullTh}{\check{U}^0_\Theta}

\newcommand{\tr}{\mathrm{tr}}

\newcommand{\Uc}{{\check{U}}}
\newcommand{\vep}{\varepsilon}

\newcommand{\Xfrak}{\mathfrak{X}}

\newcommand{\Z}{{\mathbb Z}}

\title{Braided module categories via quantum symmetric pairs}
\thanks{The early stages of this work were supported by EPSRC grant EP/K025384/1}

\author{Stefan Kolb}
\address{School of Mathematics, Statistics and Physics, Newcastle University, Newcastle upon Tyne NE1 7RU, UK}
\email{stefan.kolb@newcastle.ac.uk}


\begin{document}

\begin{abstract}
Let $\gfrak$ be a finite dimensional complex semisimple Lie algebra. The finite dimensional representations of the quantized enveloping algebra $U_q({\mathfrak g})$ form a braided monoidal category $\Oint$. We show that the category of finite dimensional representations of a quantum symmetric pair coideal subalgebra $\Bcs$ of $U_q({\mathfrak g})$ is a braided module category over an equivariantization $\Oint^{\langle \sigma \rangle}$ of $\Oint$. The braiding for $\Bcs$ is realized by a universal K-matrix which lies in a completion of $\Bcs\ot \uqg$. We apply these results to describe a distinguished basis of the center of $\Bcs$. 
\end{abstract}

\maketitle

\section{Introduction}
Quantized enveloping algebras $\uqg$, introduced by V.~Drinfeld and M.~Jimbo in the mid 80s, provide a conceptual framework for large families of solutions of the quantum Yang-Baxter equation \cite{inp-Drinfeld1}, \cite{a-Jimbo1}. These solutions are realized by a universal R-matrix, which turns suitable categories of $\uqg$-modules into braided monoidal categories, or more specifically into ribbon categories. N.~Reshetikhin and V.~Turaev developed a graphical calculus for ribbon categories \cite{a-RT90}, which is the starting point for a comprehensive theory of quantum invariants of knots and 3-manifolds \cite{b-Turaev}.
In this paper $\gfrak$ denotes a finite dimensional complex semisimple Lie algebra and the quantized enveloping algebra $\uqg$ is defined over the field $\field(q^{1/d})$ where $\field$ is a field of characteristic zero and $q^{1/d}$ denotes an indeterminate. In this case the category $\Oint$ of finite dimensional $\uqg$-modules of type 1 is a ribbon category and the Reshetikhin-Turaev calculus can be applied. 

Let $\theta:\gfrak\rightarrow \gfrak$ denote an involutive Lie algebra automorphism and let $\kfrak=\{x\in \gfrak\,|\,\theta(x)=x\}$ denote the corresponding fixed Lie subalgebra. A uniform theory of quantum symmetric pairs was formulated by G.~Letzter in \cite{a-Letzter99a}. It provides quantum group analogs $\Bcs$ of the universal enveloping algebra $U(\kfrak)$. Crucially, $\Bcs$ is a right coideal subalgebra of $\uqg$, in other words, the coproduct $\kow$ of $\uqg$ satisfies the relation
\begin{align*}
  \kow(\Bcs)\subseteq \Bcs\ot \uqg.
\end{align*} 
This means in particular that the category $\Mcs$ of finite dimensional $\Bcs$-modules is a module category over $\Oint$. Recall that the construction of the quantum symmetric pair coideal subalgebra $\Bcs$ involves a Satake diagram $(X,\tau)$ where $X$ denotes a subset of nodes of the Dynkin diagram and $\tau$ is a diagram automorphism, see \cite{a-Letzter99a}, \cite{a-Kolb14}. Moreover, let $\tau_0$ denote the diagram automorphism of $\gfrak$ corresponding to the longest element in the Weyl group. Consider the involutive diagram automorphism $\sigma=\tau\tau_0$. The group $\langle \sigma\rangle$ generated by $\sigma$ is either trivial or isomorphic to $\Z_2$. The group $\langle \sigma \rangle$ acts on $\Oint$ and the equivariantization $\Oint^{\langle \sigma \rangle}$ retains the structure of a ribbon category. Moreover, $\Mcs$ is also a module category over $\Oint^{\langle \sigma \rangle}$. 

In \cite{a-Brochier13} A.~Brochier introduced the notion of a braided module category over a braided monoidal category. He showed that Reshetikhin and Turaev's graphical calculus can be extended to braided module categories over ribbon categories. The first main result of the present paper is that quantum symmetric pairs provide a large class of examples of braided module categories.

\medskip

\noindent{\bf Theorem I} (Corollary \ref{cor:B-mod-braided}) \textit{The category $\Mcs$ is a braided module category over $\Oint^{\langle \sigma \rangle}$.}

\medskip

The notion of a braided module category has several precursors in the literature. In the late 90s T.~tom Dieck and R.~H{\"a}ring-Oldenburg introduced the notion of a braided monoidal category with a cylinder braiding \cite{a-tD98}, \cite{a-tDHO98}. If a braided module category consists of representations of a coideal subalgebra of a quasitriangular bialgebra, then it gives rise to a braided monoidal category with a cylinder braiding. B.~Enriquez gave an earlier version of the definition of a braided module category in \cite[Section 4.3]{a-Enriquez07} which seems to have been the inspiration for \cite{a-Brochier13}. Structures similar to braided module categories also appeared in the physics literature, see for example \cite[Section 3.2]{a-Baseilhac05b}.

The proof of Theorem I builds on our previous work \cite{a-BalaKolb15p} which established that quantum symmetric pairs provide examples of braided monoidal categories with a cylinder braiding as defined in \cite{a-tD98}, \cite{a-tDHO98}. The cylinder braiding in \cite{a-BalaKolb15p} is realized by a universal K-matrix $\cK_{\BK}$ (denoted by $\cK$ in \cite{a-BalaKolb15p}) which lies in a completion $\sU$ of $\uqg$. The element $\cK_\BK$ satisfies the commutation relation
\begin{align}\label{eq:KBKb}
  \sigma(b) \cK_\BK = \cK_{\BK} b \qquad \mbox{for all $b\in \Bcs$}
\end{align}
and provides solutions of a $\sigma$-twisted version of the reflection equation in any representation $V\in Ob(\Oint)$.
The construction of the universal K-matrix $\cK_\BK$ in \cite{a-BalaKolb15p} followed the first steps in a recent program of canonical bases for quantum symmetric pairs proposed in \cite{a-BaoWang13p}, see also \cite{a-BaoWang16p}.

In the present paper, to avoid the twist by $\sigma$ in \eqref{eq:KBKb} and in the reflection equation, we consider the semidirect product $U_\sigma=\uqg\rtimes \field(q^{1/d})\langle \sigma \rangle$. Then $\Oint^{\langle \sigma \rangle}$ is nothing but the category of finite dimensional $U_\sigma$-modules which belong to $\Oint$ after restriction to $\uqg$. We consider the element $\cK=\cK_\BK\sigma$ which lies in a completion of $U_\sigma$. The element $\cK$ commutes with all $b\in \Bcs$ and provides solutions of the usual reflection equation in any representation $V\in Ob(\Oint^{\langle\sigma \rangle})$. The structure of a braided module category on $\Mcs$ is realized by an element $\Kmat$ in a completion $\sB^{(2)}$ of $\Bcs\ot U_\sigma$. The relation between $\Kmat$ and $\cK$ is given by 
\begin{align}\label{eq:Kmat-def}
  \Kmat=\Rmat_{21}(1\ot \cK) \Rmat
\end{align}
where $\Rmat$ denotes the universal R-matrix for $\uqg$. The crucial new result is that the right hand side of \eqref{eq:Kmat-def} does indeed belong to $\sB^{(2)}$ and hence acts naturally on $M\ot V$ for all $M\in Ob(\Mcs)$ and $V\in Ob(\Oint^{\langle \sigma \rangle})$. The main technical ingredient needed to prove that $\Kmat\in \sB^{(2)}$ is a generalization of results proved for quantum symmetric pairs of type AIII/IV  with $X=\emptyset$ in \cite[Section 3]{a-BaoWang13p}. We call $\Kmat$ the 2-tensor universal K-matrix for $\Bcs$.

In the second part of the paper we apply the 2-tensor universal K-matrix $\Kmat$ to revisit the investigation of the center $Z(\Bcs)$ of $\Bcs$ from \cite{a-KL08}. In \cite{a-KL08} we showed that $Z(\Bcs)$ has a direct sum decomposition into distinguished one-dimensional subspaces labeled by certain dominant integral weights. We showed, moreover, that $Z(\Bcs)$ is a polynomial ring in $\rank(\kfrak)$ variables. The proofs in \cite{a-KL08} are technical. They rely on the interplay between several filtrations of $\uqg$, the explicit form of the generators of $\Bcs$, and the structure of the locally finite part of $\uqg$ described in \cite{a-JoLet2}. At the end of his $\mathrm{MathSciNet}^{\copyright}$ review of \cite{a-KL08}, A.~Joseph asked if one can make effective use of the Hopf dual $\cA$ of $\uqg$ in the construction of the center $Z(\Bcs)$. Such a construction is well known for the center of $\uqg$, see for example \cite[Proposition 3.3]{a-Drin90-almost}, \cite{a-Baumann98} and references therein. Guided by the graphical calculus for braided module categories, we use the 2-tensor universal K-matrix $\Kmat$ to answer Joseph's question, and to give new proofs of all the results in \cite{a-KL08}. More explicitly, let $\cA_\sigma$ denote the Hopf dual of $U_\sigma$ and let $\cA_\sigma^\inv\subset \cA_\sigma$ denote the subalgebra of elements invariant under the dual of the left adjoint action of $U_\sigma$ on itself.
Let $P^+$ denote the set of dominant integral weights for $\gfrak$ and for each $\lambda\in P^+$ let $V(\lambda)$ denote the corresponding simple $\uqg$-module.
If $\lambda\in P^+$ satisfies $\sigma(\lambda)=\lambda$ then $V(\lambda)$ has a uniquely determined $U_\sigma$-module structure such that any highest weight vector of $V(\lambda)$ is fixed by $\sigma$. We denote the resulting $U_\sigma$-module by $V(\lambda)_+$ and we let $\tr_{V(\lambda)_+,\sigma,q}\in \cA^\inv_\sigma$ be the corresponding quantum trace. We obtain the following description of the center $Z(\Bcs)$. Let $v\in \sU$ denote the ribbon element for $\uqg$.

\medskip
  \noindent {\bf Theorem II} (Theorem \ref{thm:Z(B)-basis}) \textit{The map
  \begin{align}\label{eq:ltilvK}
    \ltil_{(1\ot \rib) \cdot \Kmat}:\cA^\inv_{\sigma} \rightarrow Z(\Bcs)
  \end{align}
  given by pairing elements in $\cA^\inv_\sigma$ with the second tensor factor of $(1\ot \rib)\cdot \Kmat$ is a surjective algebra homomorphism. Moreover, the set
  \begin{align}\label{eq:vKbasis}
    \{\ltil_{(1\ot \rib) \cdot \Kmat}(\tr_{V(\lambda)_+,\sigma,q})\,|\,\lambda\in P^+, \sigma(\lambda)=\lambda\}
  \end{align}
  is a basis of $Z(\Bcs)$.}

\medskip

There is some freedom in the choice of axioms for a braided module category. The appearance of the ribbon element $\rib$ in the map \eqref{eq:ltilvK} is solely due to the particular choice made in \cite{a-Brochier13}, see Remark \ref{rem:differentK}.

The kernel of the map \eqref{eq:ltilvK} is explicitly determined in Corollary \ref{cor:ZBZ}. Define $P_{Z(\Bcs)}=\{\lambda\in P^+\,|\,\sigma(\lambda)=\lambda\}$ and set $b_\lambda=\ltil_{(1\ot \rib) \cdot \Kmat}(\tr_{V(\lambda)_+,\sigma,q})$ for all $\lambda\in P_{Z(\Bcs)}$. It turns out that the elements of the distinguished basis \eqref{eq:vKbasis} satisfy the relation
\begin{align}\label{eq:LR}
  b_\lambda b_\mu = \sum_{\nu\in P_{Z(\Bcs)}} \eta_{\lambda,\mu}^\nu b_\nu \qquad \mbox{for all $\lambda,\mu\in P_{Z(\Bcs)}$}
\end{align}
where the coefficients $\eta_{\lambda,\mu}^\nu\in \N_0$ can be read off from the tensor product decomposition of $V(\lambda)_+\ot V(\mu)_+$ in $\Oint^{\langle \sigma \rangle}$, see Corollary \ref{cor:ZBZ}. If $\sigma=\id$ then $\eta_{\lambda,\mu}^\nu$ are nothing but the Littlewood-Richardson coefficients for the semisimple Lie algebra $\gfrak$. A special instance of relation \eqref{eq:LR} for a distinguished basis of $Z(\Bcs)$ was previously observed in \cite[Theorem 5.13]{a-KolbStok09}.
In the final Section \ref{sec:tqt} we observe that $Z(\Bcs)$ can also be identified with a subalgebra $\cA^{\inv,\sigma}\subset\cA$ consisting of all elements which are invariant under a $\sigma$-twisted adjoint action.

The main body of the paper is organized in three sections. In Section \ref{sec:BMC} we recall the notion of a braided module category and relate it to the concepts in \cite{a-tD98}, \cite{a-tDHO98}, \cite{a-Enriquez07}, as well as our previous work \cite{a-BalaKolb15p}. In Section \ref{sec:univ-K-qsp} we construct the universal K-matrix $\Kmat$ and show that it belongs to $\sB^{(2)}$. This allows us to prove Theorem I. In Section \ref{sec:center} we investigate the center $Z(\Bcs)$ from scratch. Compared with \cite{a-KL08} we provide a much streamlined presentation and eventually prove Theorem II.

\medskip

\noindent{\bf Acknowledgments.} The author is grateful to Pascal Baseilhac who insisted in November 2015 that there should be a 2-tensor universal K-matrix in (a completion of) $\Bcs\ot \uqg$. The author is much indebted to an anonymous referee who suggested to include the factor $\sigma$ in the definition of the 1-tensor universal K-matrix $\cK$ and to consider $\Mcs$ as a module category over $\Oint^{\langle \sigma \rangle}$ instead of $\Oint$. This suggestion led to a substantial revision of the paper and made it possible to interpret the multiplicative behavior of the distinguished basis of $Z(\Bcs)$ in full generality. Part of this work was done while the author was visiting the Bereich Algebra und Zahlentheorie at Universit{\"a}t Hamburg in February 2017. The author is grateful for the hospitality and the good working conditions. 

\section{Braided module categories}\label{sec:BMC}
In this introductory section we recall the notion of a braided module category over a braided monoidal category following \cite{a-Brochier13}. Let $H$ be a bialgebra. Representations of right $H$-comodule algebras provide examples of module categories. We formulate the notion of a quasitriangular comodule algebra over a quasitriangular bialgebra in Section \ref{sec:quasi-comod}. In Section \ref{sec:quasi-coid} we reformulate this notion for right coideal subalgebras of $H$. Representations of quasitriangular right $H$-comodule algebras provide examples of braided module categories.
\subsection{Strict braided module categories}
Let $(\cV,\ot,\mathbbm{1})$ be a strict monoidal category \cite[1.1]{b-Turaev}. A strict right module category over $\cV$ consists of a category $\cM$ together with a functor $\boxtimes:\cM \times \cV \rightarrow \cM$ such that
\begin{align}\label{eq:modCat1}
  M\boxtimes \mathbbm{1} = M, \qquad (M \boxtimes V) \boxtimes W = M \boxtimes (V \otimes W)
\end{align}
for any object $M$ in  $\cM$ and objects $V, W$ in $\cV$ and
\begin{align}\label{eq:modCat2}
  f\boxtimes \id_{\mathbbm{1}}=f, \qquad (f\boxtimes g) \boxtimes h=f\boxtimes (g\otimes h)
\end{align}
for any morphism $f$ in $\cM$ and morphisms $g,h$ in $\cV$. 
Recall that a monoidal category $\cV$ is called braided if it is equipped with a braiding
\begin{align*}
  c=\{c_{V,W}:V\ot W \rightarrow W\ot V\}_{V,W\in Ob(\cV)}
\end{align*}
which satisfies the relations given in \cite[1.2]{b-Turaev}.
\begin{defi}[{\cite[5.1]{a-Brochier13}}] \label{def:braidedMod}
  Let $(\cV,\ot,\mathbbm{1},c)$ be a (strict) braided monoidal category. A (strict) braided module category over $(\cV,\ot,\mathbbm{1},c)$ is a (strict) module category $(\cM,\boxtimes)$ over $\cV$ equipped with a natural automorphism $e\in \Aut(\boxtimes)$ which satisfies the relations
  \begin{align}
    e_{M\boxtimes V, W}&=(\id_M\boxtimes c_{W,V})(e_{M,W}\boxtimes \id_V)(\id_M\boxtimes c_{V,W}),\label{eq:braidedMod1}\\
    e_{M, V\otimes W}&=(\id_M\boxtimes c_{W,V})(e_{M,W}\boxtimes \id_V)(\id_M\boxtimes c_{V,W})(e_{M,V}\boxtimes \id_W)\label{eq:braidedMod2}
  \end{align}
  for all objects $M\in Ob(\cM)$ and $V,W\in Ob(\cV)$. 
\end{defi}
\begin{rema}
In view of \eqref{eq:braidedMod1}, we can rewrite relation \eqref{eq:braidedMod2} as
\begin{align*}
    e_{M, V\otimes W}&=e_{M\boxtimes V, W}(e_{M,V}\boxtimes \id_W).
\end{align*}
This is the form in which \eqref{eq:braidedMod2} appears in \cite{a-Brochier13}.
Moreover, the naturality of $e$ implies that \eqref{eq:braidedMod2} is equivalent to
\begin{align}\label{eq:braidedMod2'}
    e_{M, V\otimes W}&=(e_{M,V}\boxtimes \id_V)(\id_M\boxtimes c_{W,V})(e_{M,W}\boxtimes \id_V)(\id_M\boxtimes c_{V,W}).
\end{align}
The equality between \eqref{eq:braidedMod2} and \eqref{eq:braidedMod2'} implies that for any $M\in Ob(\cM)$ and any $V\in Ob(\cV)$ the maps $e_{M,V}$ and $c_{V,V}$ provide a repesentation of the annular braid group (or braid group of type $B_n$) on $M\boxtimes V^{\ot n}$.
\end{rema}
\begin{rema}
  For general module categories the strictness conditions \eqref{eq:modCat1} and \eqref{eq:modCat2} are replaced by natural isomorphisms $\rho_M: M\boxtimes \mathbbm{1} \rightarrow M$ and $\alpha_{M,V,W}:(M \boxtimes V) \boxtimes W \rightarrow M \boxtimes (V \otimes W)$ which satisfy a triangle and a pentagon condition, see for example \cite[(2.1), (2.2)]{a-tD98}. The general definition of a braided module category over a not necessarily strict monoidal category can be found in \cite[5.1]{a-Brochier13}.  In the following we will consider categories of modules over algebras and both $\otimes$ and $\boxtimes$ are ordinary tensor products of vector spaces. These categories are not strict, but we suppress the associativity constraints for convenience, canonically identifying multiple tensor products with different brackets.
\end{rema}
\begin{rema}\label{rem:differentK}
  As pointed out in \cite[Remark 3.6]{a-BBJ16p}, relation \eqref{eq:braidedMod2} is but one of an infinite family of possible axioms for a braided module category. In particular, it may make sense to replace \eqref{eq:braidedMod2} by the axiom 
  \begin{align}
    e_{M, V\otimes W}&=(\id_M\boxtimes c_{V,W}^{-1})(e_{M,W}\boxtimes \id_V)(\id_M\boxtimes c_{V,W})(e_{M,V}\boxtimes \id_W).\label{eq:braidedMod2-rib}
  \end{align}
If $\cV$ is a ribbon category with a twist $\rib=\{\rib_{V}:V\rightarrow V\}_{V\in Ob(\cV)}$ which satisfies $\rib_{V\ot W}=c_{W,V}c_{V,W}(\rib_V\ot \rib_W)$, see \cite[1.4]{b-Turaev}, then \eqref{eq:braidedMod2-rib} is obtained from \eqref{eq:braidedMod2} by replacing $e_{M,V}$ by $(1\ot v^{-1})e_{M,V}$. 
\end{rema}
\begin{rema}
  Relations \eqref{eq:braidedMod1} and \eqref{eq:braidedMod2} have a graphical interpretation. Assume that $\cV$ is a ribbon category, see \cite[1.4]{b-Turaev}. Then there exists a graphical calculus for morphisms in $\cV$, see \cite{a-RT90}, \cite[Theorem 2.5]{b-Turaev}, given by a functor from the category of $\cV$-colored ribbon tangles to the category $\cV$. Within this calculus the commutativity isomorphism $c_{V,W}$ is represented by the (downwards oriented) diagram in Figure \ref{fig:R-graph}.
 \begin{figure}[H]
  \centering
    \begin{align*} 
 \vcenter{\hbox{\includegraphics[height=2cm]{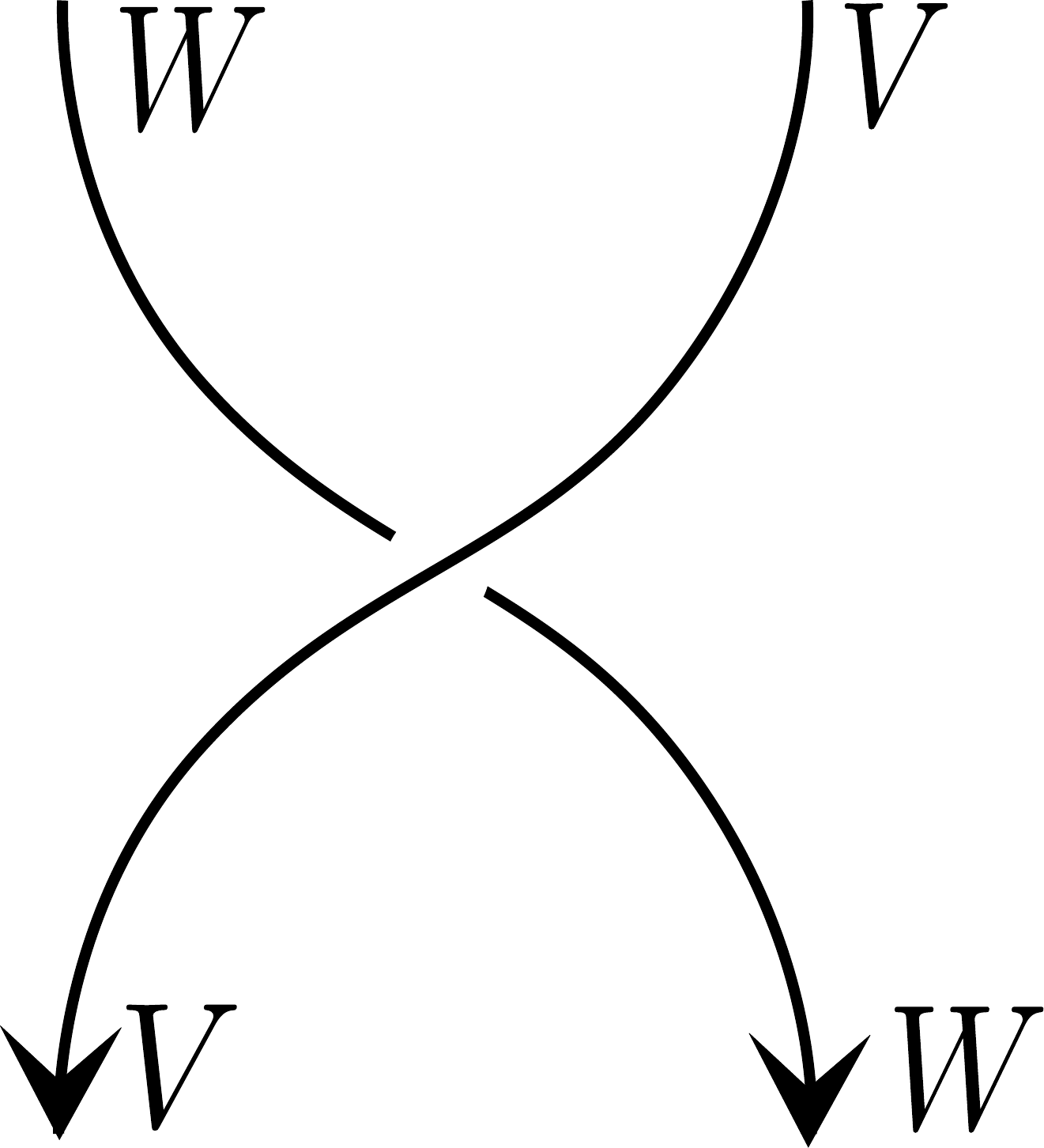}}} 
\end{align*}
  \caption{}
  \label{fig:R-graph}
\end{figure} 
The Reshetikhin-Turaev functor has been extended to braided module categories in \cite[Theorem 5.3]{a-Brochier13}. If $\cM$ is a braided module category over $\cV$ then the braiding $e_{M,V}$ is graphically represented by Figure \ref{fig:K-graph}.
\begin{figure}[H]
\begin{align*}
  \vcenter{\hbox{\includegraphics[height=2cm]{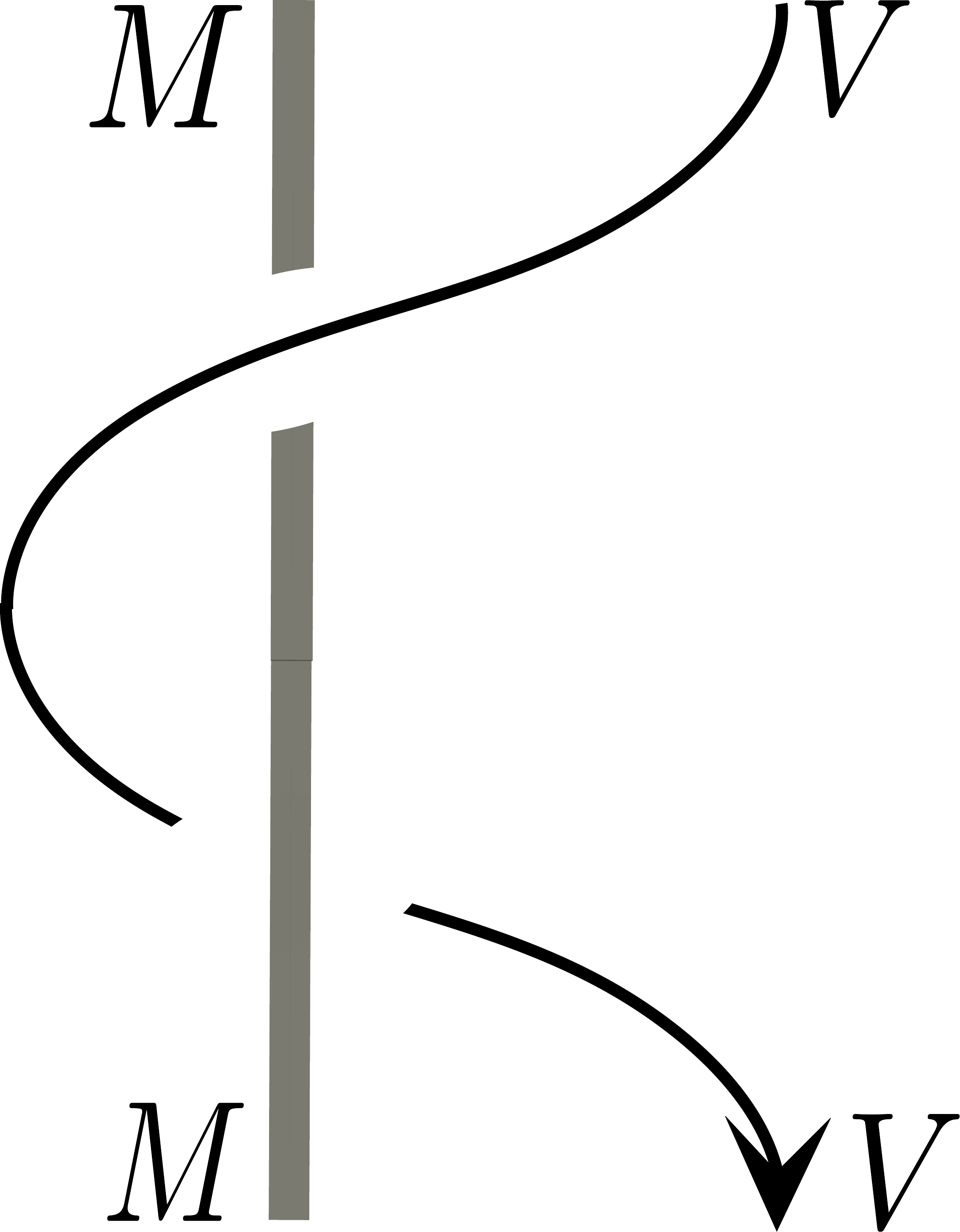}}}  
\end{align*}
\caption{} 
\label{fig:K-graph}
\end{figure}
Here the thick line represents a fixed pole which is colored by an object in $\cM$. The other line represents a $\cV$ colored ribbon which winds around the pole. Relations \eqref{eq:braidedMod1} and \eqref{eq:braidedMod2} can now be depicted by Figure \ref{fig:br-mod-1} and Figure \ref{fig:br-mod-2}, respectively.
\begin{figure}[H]
\begin{align*}
     \vcenter{\hbox{\includegraphics[height=2cm]{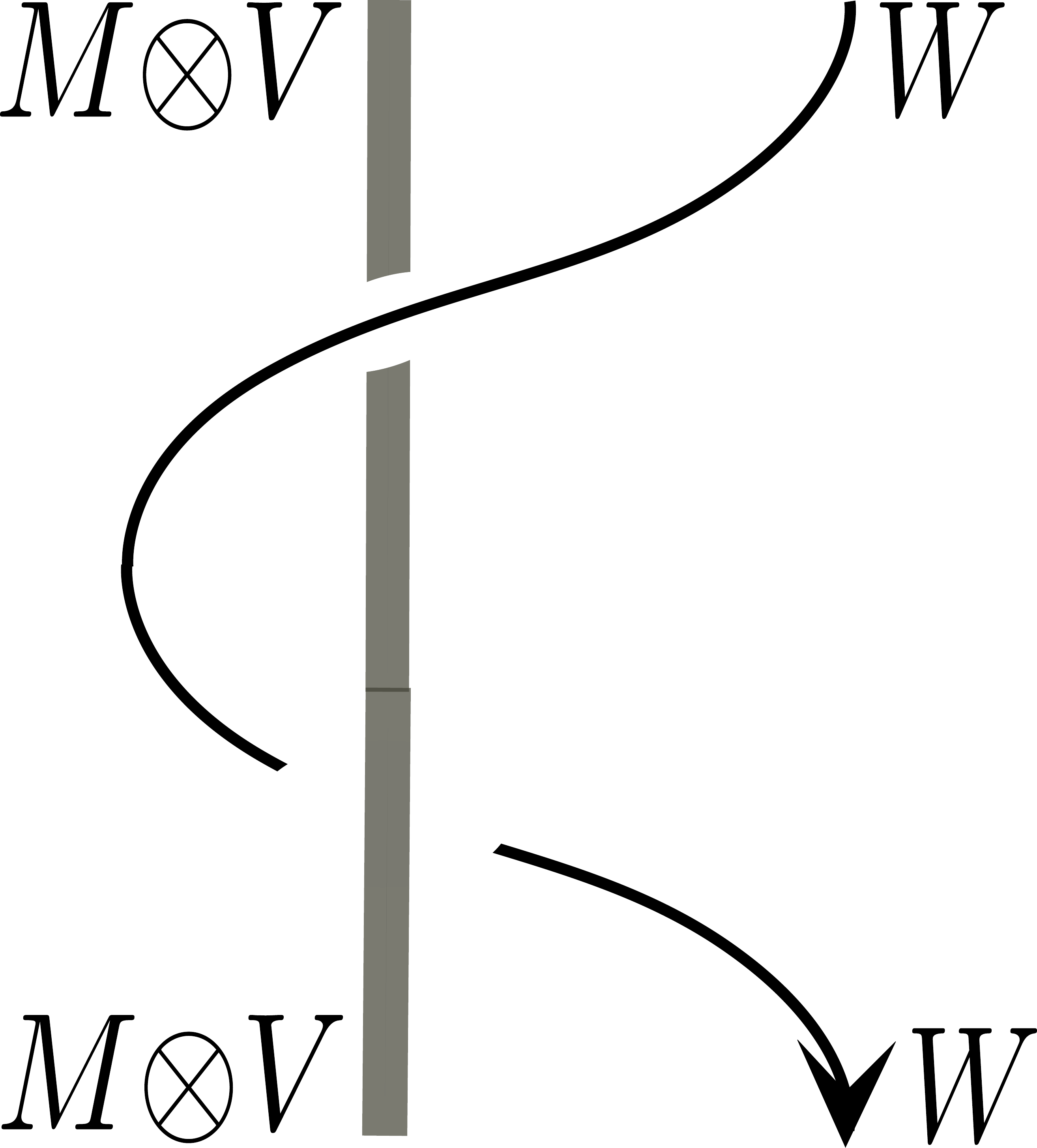}}} &= \quad
     \vcenter{\hbox{\includegraphics[height=2cm]{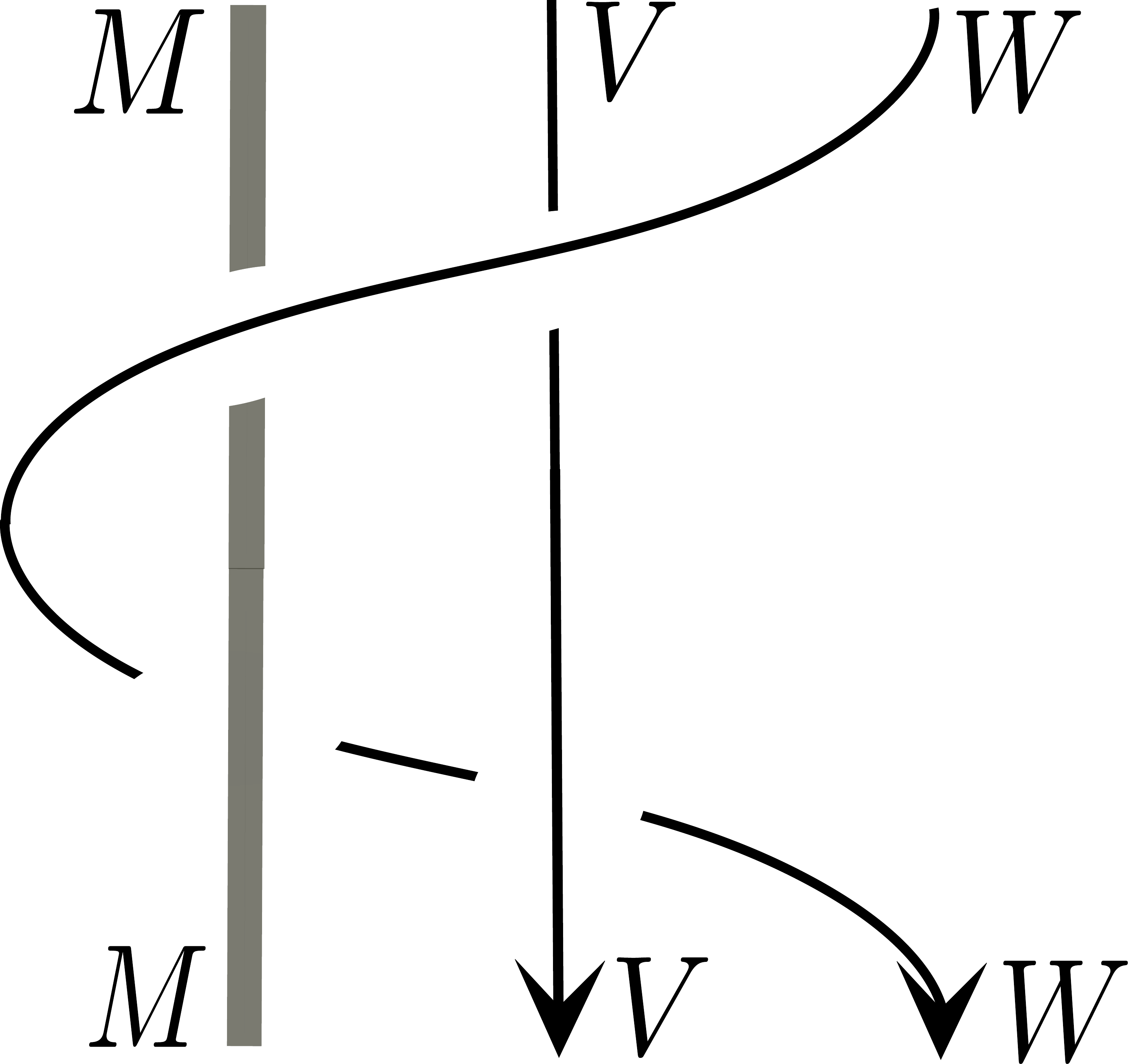}}} = \quad
     \vcenter{\hbox{\includegraphics[height=2cm]{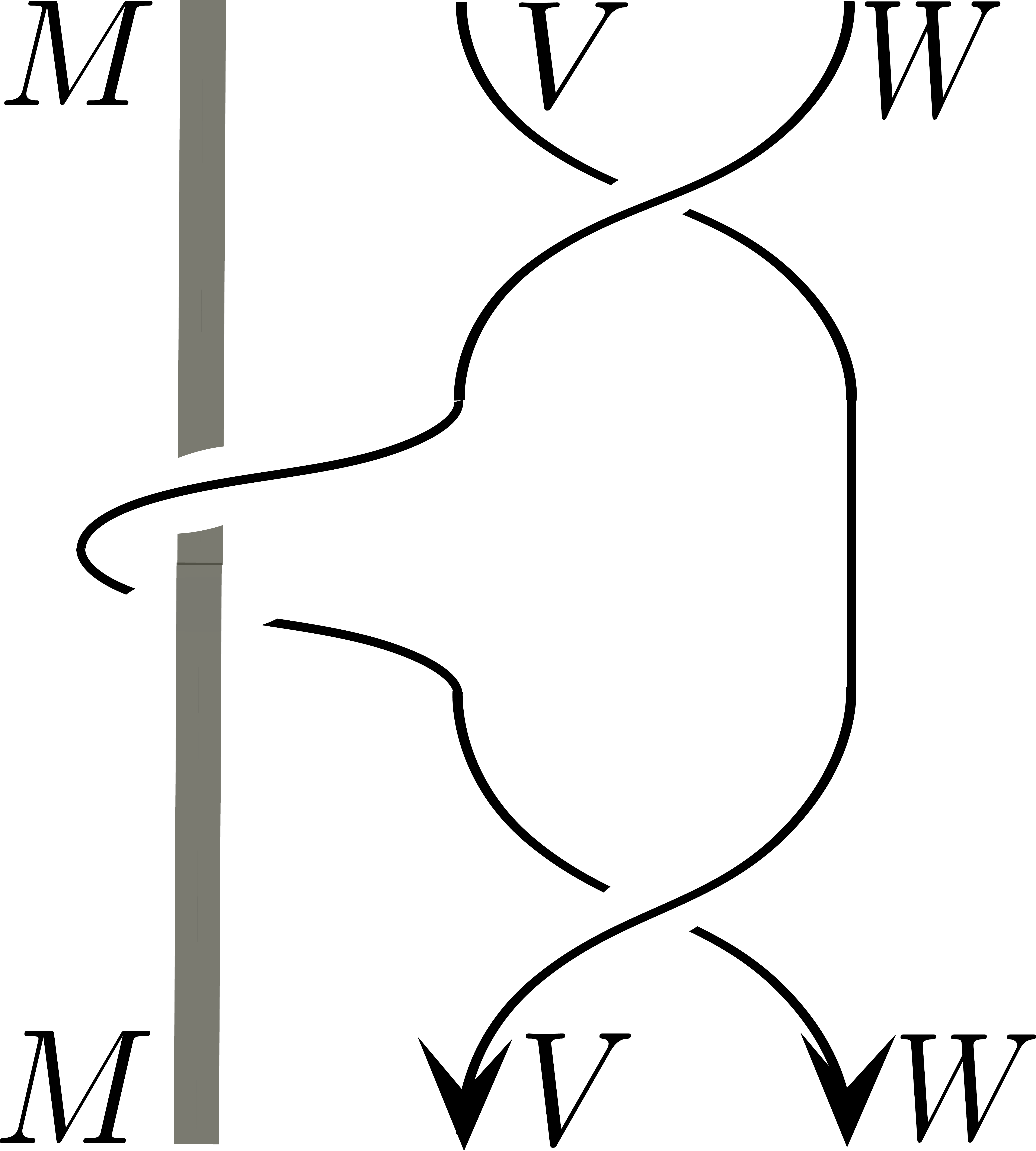}}} 
\end{align*}
\caption{}
\label{fig:br-mod-1}
\end{figure}
\begin{figure}[H]
\begin{align*}
     \vcenter{\hbox{\includegraphics[height=2cm]{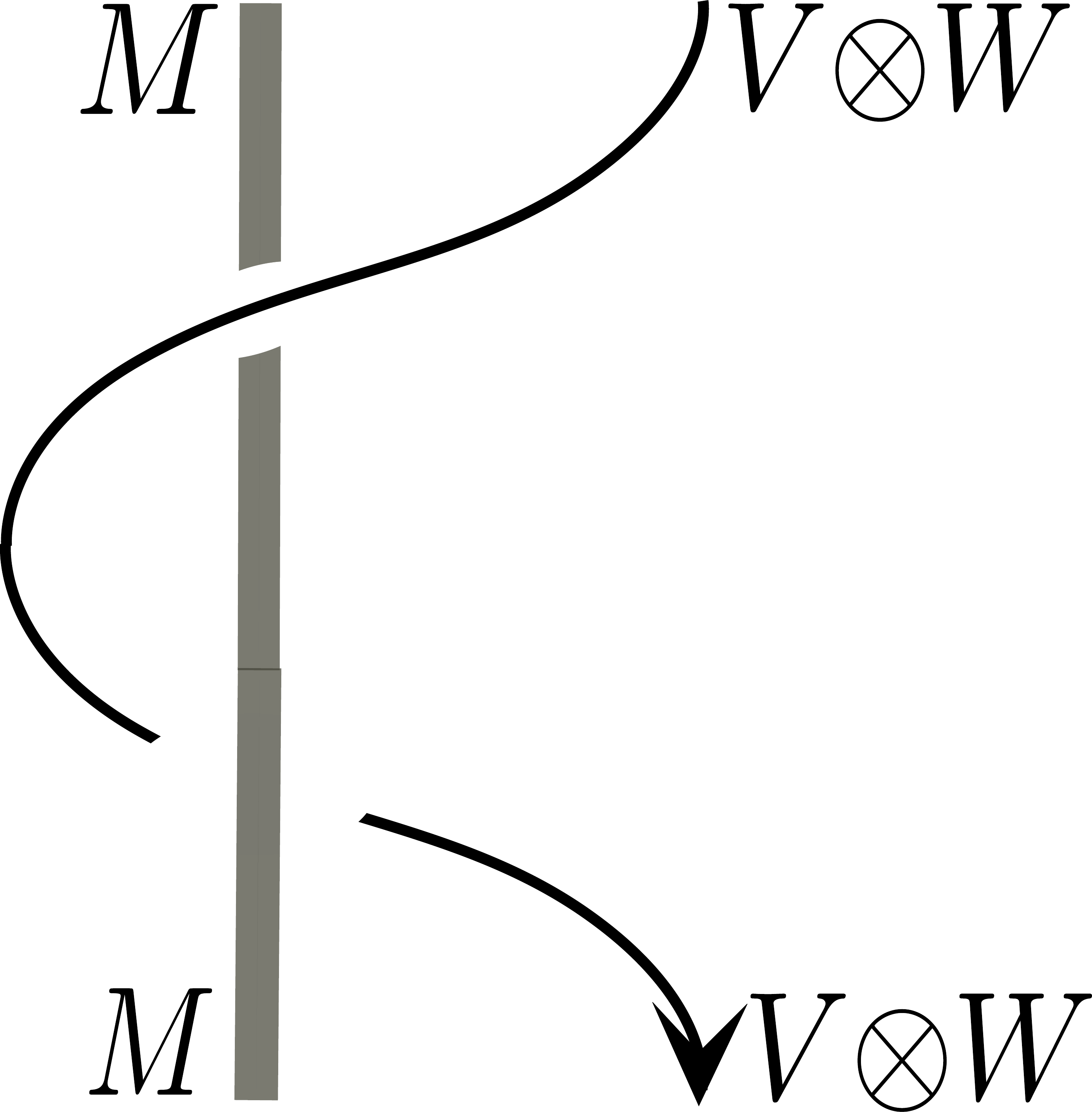}}} &= \quad
     \vcenter{\hbox{\includegraphics[height=2cm]{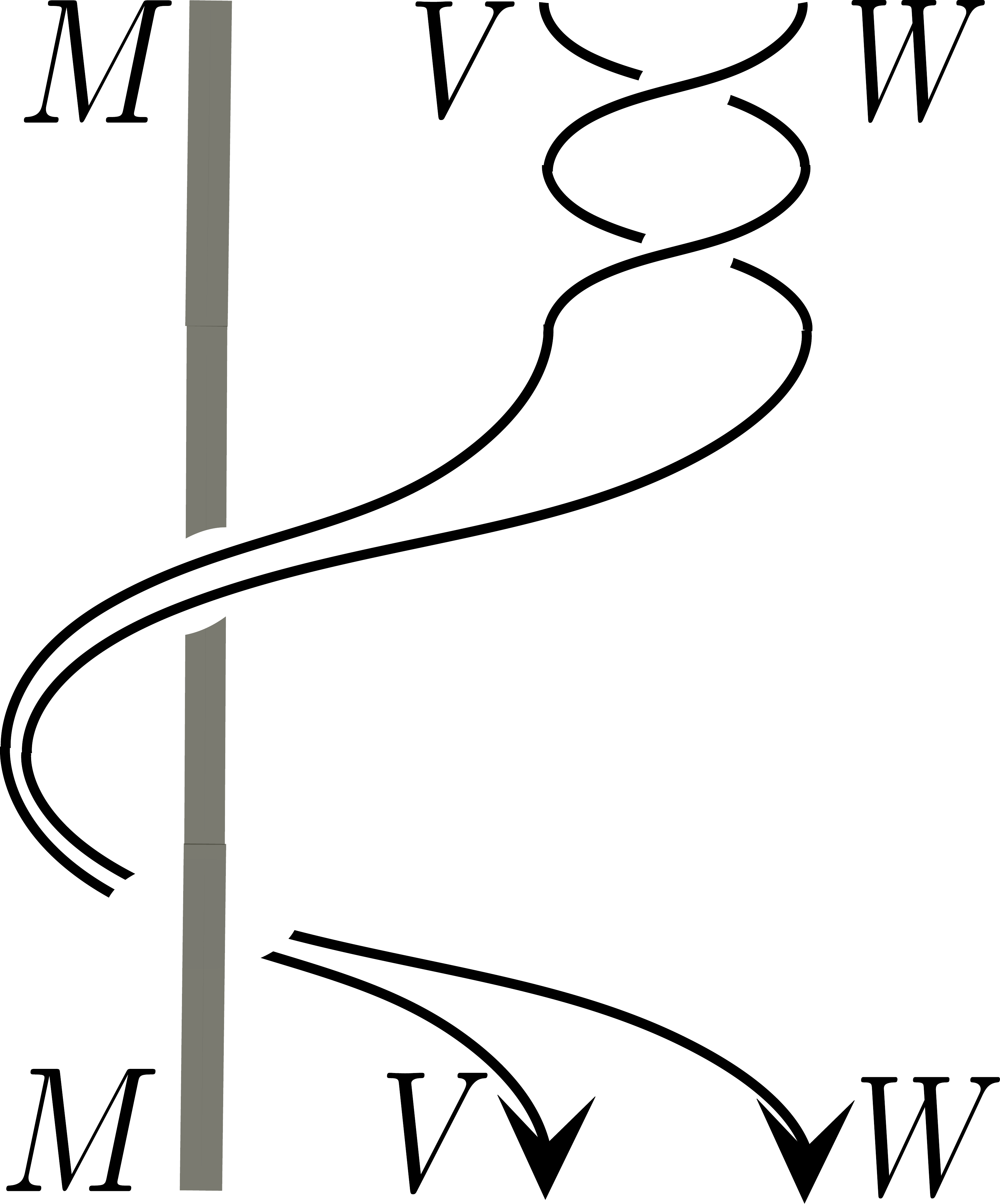}}} =\quad
     \vcenter{\hbox{\includegraphics[height=2cm]{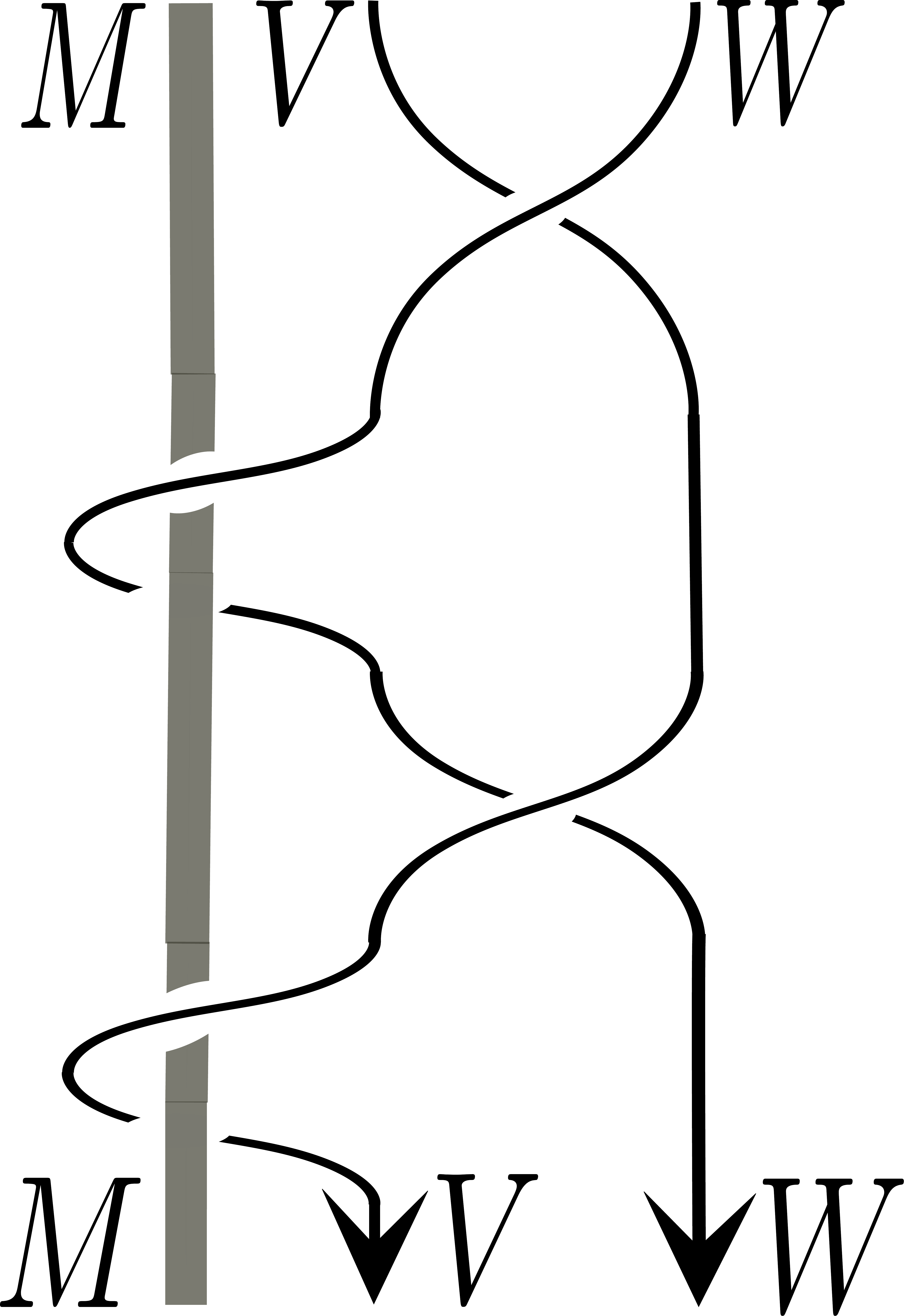}}}
\end{align*}
\caption{}
\label{fig:br-mod-2}
\label{fig:Knat2}
\end{figure}
The additional twist in the middle diagram of Figure \ref{fig:Knat2} comes from the convention that a ribbon circles around the pole in the cylinder plane, that is, with a fixed side facing the pole. If one replaces axiom \eqref{eq:braidedMod2} by axiom \eqref{eq:braidedMod2-rib} then the twist in the middle diagram of Figure \ref{fig:Knat2} disappears. Geometrically this corresponds to the convention that the ribbon lies in the paper plane when circling the pole, see \cite[Definition 3.4, Remark 3.6]{a-BBJ16p}.  
\end{rema}
\subsection{Quasitriangular comodule algebras}\label{sec:quasi-comod}
For later reference we recall the definition of a quasitriangular bialgebra following \cite[Definition VIII.2.2]{b-Kassel1}. All through this paper we will make use of the leg notation for tensor products explained in \cite[p.~172]{b-Kassel1}.
\begin{defi}\label{def:universalR}
  A bialgebra $(H,\kow,\vep)$ is called quasitriangular if there exists an invertible element $\Rmat\in H\ot H$ which satisfies the following relations:
  \begin{enumerate}
    \item $\Rmat\kow(x) = \kow^\op(x) \Rmat$ for all $x\in H$, \label{itemR1}
    \item $(\kow \ot \id)(\Rmat)=\Rmat_{13} \Rmat_{23}$,  \label{itemR2}\label{R2}
    \item $(\id \ot\kow)(\Rmat)=\Rmat_{13} \Rmat_{12}$. \label{itemR3}\label{R3}
  \end{enumerate}
 In this case the element $\Rmat$ is called a universal R-matrix for $H$. 
\end{defi}
Properties \eqref{itemR1} and \eqref{itemR2} in Definition \ref{def:universalR} imply the quantum Yang-Baxter equation 
\begin{align}\label{eq:qYB1}
  \Rmat_{12} \Rmat_{13} \Rmat_{23} = \Rmat_{23} \Rmat_{13} \Rmat_{12}.
\end{align}
Alternatively, Equation \eqref{eq:qYB1} also follows from \eqref{itemR1} and \eqref{itemR3}. If $(H,\kow,\vep,\Rmat)$ is a quasitriangular bialgebra with universal R-matrix $\Rmat$ then the category $\cV$ of $H$-modules is a braided monoidal category, see \cite[Proposition XIII.1.4]{b-Kassel1}. For $H$-modules $V,W$ the braiding $c_{V,W}:V\ot W\rightarrow W\ot V$ is given by $c_{V,W}=\flip\circ \Rmat$ where $\flip$ denotes the flip of tensor factors.

To obtain braided module categories in a similar way, we introduce the notion of a quasitriangular comodule algebra.
\begin{defi}\label{def:K-matrix}
  Let $(H,\kow,\vep,\Rmat)$ be a quasitriangular bialgebra. A right $H$-comodule algebra $B$ with coaction $\kow_B:B\rightarrow B\ot H$ is called quasitriangular if there exists an invertible element $\Kmat\in B\ot H$ which satisfies the following relations:
  \begin{enumerate}
    \item[(K1)] $\Kmat\kow_B(b)= \kow_B(b)\Kmat$ for all $b\in B$,\label{item1}
    \item[(K2)] $(\kow_B\ot \id)(\Kmat) = \Rmat_{32}\Kmat_{13}\Rmat_{23}$,\label{item2}
    \item[(K3)] $(\id \ot \kow) (\Kmat)=\Rmat_{32}\Kmat_{13} \Rmat_{23} \Kmat_{12}$. \label{item3}  
  \end{enumerate}
  Here we label the tensor components of $B\ot H\ot H$ by $1$, $2$, and $3$. The element $\Kmat$ is called a universal K-matrix for the $H$-comodule algebra $B$.  
\end{defi}
In the setting of the above definition the category $\cM$ of $B$-modules is a braided module category over the category $\cV$ of $H$-modules. The tensor product $\boxtimes$ is given by the usual tensor product $M\ot V$ of a $B$-module $M$ with an $H$-module $V$. The braiding $e_{M,V}:M\otimes V \rightarrow M\otimes V$ for $\cM$ is given by multiplication by $\Kmat$.
\begin{rema}\label{rem:refl-alg}
  A precursor of Definition \ref{def:K-matrix} appeared in \cite[Definition 4.1]{a-Enriquez07} under the name \textit{reflection algebra}. A reflection algebra over a quasitriangular bialgebra $H$ is a $H$-comodule algebra $B$ with an element $\Kmat\in B\ot H$ which satisfies (K1) and (K2) but not necessarily (K3). Observe that conditions (K1) and (K2) are preserved under multiplication of $\Kmat$ by scalars, while condition (K3) is not preserved. More generally, quasi-reflection algebras over quasi-Hopf algebras are considered in \cite{a-Enriquez07}.
\end{rema}
\subsection{Quasitriangular coideal subalgebras}\label{sec:quasi-coid}
Assume now that $B$ is a right coideal subalgebra of the quasitriangular bialgebra $H$. In other words, $B\subseteq H$ is a subalgebra for which $\kow(B)\subset B\ot H$. In this case Definition \ref{def:K-matrix} can be reformulated. Indeed, given a universal K-matrix $\Kmat$ for $B$ we can apply the counit $\vep$ to the first tensor factor of $K$ to obtain an element
\begin{align*}
  \cK=(\vep\ot \id)(\Kmat)\in H.
\end{align*} 
It follows from the properties (K1) - (K3) in Definition \ref{def:K-matrix} that $\cK$ has the following properties:
\begin{enumerate}
\setcounter{enumi}{3}
 \item[(K$1'$)]  $\cK b=b\cK$ for all $b\in B$,
 \item[(K$2'$)]  $\Rmat_{21}\cK_2 \Rmat_{12} \in B\ot H$,
 \item[(K$3'$)]  $\kow(\cK)=\Rmat_{21} \cK_{2} \Rmat_{12}\cK_{1}$.
\end{enumerate}
Indeed, the three above relations are obtained by applying the counit to the first tensor factor in each of the relations (K1) - (K3), respectively. In the weaker setting of reflection algebras discussed in Remark \ref{rem:refl-alg}, the following lemma is stated in \cite[Remark 4.2]{a-Enriquez07}.
\begin{lem}\label{lem:KcalK}
  Let $B$ be a right coideal subalgebra of a quasitriangular bialgebra $H$. 
  
  1) If $B$ is quasitriangular with universal K-matrix  $\Kmat\in B\ot H$, then the invertible element $\cK=(\vep\ot \id)(\Kmat)$ satisfies properties {\upshape (K$1'$)-(K$3'$)} above. In this case $\Kmat=\Rmat_{21}\cK_2 \Rmat_{12}$.
   
  2) Conversely, if an invertible element $\cK\in H$ satisfies the properties {\upshape (K$1'$)-(K$3'$)} above, then $B$ is quasitriangular with universal K-matrix $\Rmat_{21}\cK_2 \Rmat_{12} \in B\ot H$.
\end{lem}
\begin{proof}
  To verify the second part of statement 1) one calculates
    \begin{align*}
      \Kmat=(\vep\ot \id\ot \id) (\kow\ot \id)(\Kmat) \stackrel{\mathrm{(K2)}}{=}  (\vep\ot \id\ot \id)(\Rmat_{32} \Kmat_{13} \Rmat_{32})=\Rmat_{21}\cK_2 \Rmat_{12}.
    \end{align*}
  To verify statement 2), assume that $\cK\in H$ satisfies the properties (K$1'$)-(K$3'$) above. As $\cK_2$ commutes with $\kow^\op(b)$ by the coideal property, one obtains property (K1) for $\Kmat=\Rmat_{21}\cK_2 \Rmat_{12}$ from  property \eqref{itemR1} in Definition \ref{def:universalR}. 
Moreover, the relation
  \begin{align*}
    (\kow\ot \id)(\Rmat_{21}\cK_2 \Rmat_{12})&= \Rmat_{32} \Rmat_{31} \cK_3 \Rmat_{13} \Rmat_{23}
  \end{align*}
implies that property (K2) of  Definition \ref{def:K-matrix} holds for $\Kmat=\Rmat_{21}\cK_2 \Rmat_{12}$.
Finally, using
  \begin{align*}
    (\id\ot \Delta)(\Rmat_{21})&=\Rmat_{21}\Rmat_{31},
  \end{align*}
  Property (K$3'$), and the quantum Yang-Baxter equation \eqref{eq:qYB1} one calculates 
  \begin{align*}
    (\id \ot \Delta)(\Rmat_{21}\cK_2 \Rmat_{12})&=\Rmat_{21} \Rmat_{31} \Rmat_{32} \cK_{3} \Rmat_{23}\cK_{2}\Rmat_{13}\Rmat_{12}\\
      &=\Rmat_{32} \Rmat_{31} \Rmat_{21} \cK_{3} \Rmat_{23}\Rmat_{13}\cK_2\Rmat_{12}\\
      &=\Rmat_{32} \Rmat_{31} \cK_3 \Rmat_{21}\Rmat_{23}\Rmat_{13}\cK_2\Rmat_{12}\\
      &=\Rmat_{32} \Rmat_{31} \cK_3 \Rmat_{13}\Rmat_{23}\Rmat_{21}\cK_2\Rmat_{12}.
  \end{align*}
  This shows that $\Kmat=\Rmat_{21}\cK_2 \Rmat_{12}$ satisfies property (K3) of Definition \ref{def:K-matrix}.
\end{proof}
Lemma \ref{lem:KcalK} shows that the information contained in $\Kmat$ and $\cK$ is equivalent and hence they both deserve to be called a universal K-matrix for $B$. For this reason we work with the following terminology.
\begin{conv}
  Let $B$ be a right coideal subalgebra of a quasitriangular bialgebra H. An element $\cK\in H$ satisfying {\upshape (K$1'$), (K$2'$)}, and {\upshape (K$3'$)} is called a 1-tensor universal K-matrix for B. An element $\Kmat$ satisfying {\upshape(K$1$), (K$2$)}, and {\upshape (K$3$)} is called a 2-tensor universal K-matrix for $B$. We drop the qualifiers 1-tensor and 2-tensor if it is clear from the context whether we consider $\cK$ or $\Kmat$.
\end{conv}  
\begin{rema}
  Lemma \ref{lem:KcalK} shows that Definition \ref{def:K-matrix} is a refinement of the definition of a braided coideal subalgebra given previously in \cite[Definition 4.10]{a-BalaKolb15p}. Indeed, condition (K$2'$) was missing in our previous attempt to define a quasitriangular (braided) coideal subalgebra. The new definition \ref{def:K-matrix} is preferable, because the universal K-matrix is now specific to the coideal subalgebra $B$. In the definition given in \cite[Definition 4.10]{a-BalaKolb15p} a universal K-matrix for $B$ would also be a universal K-matrix for any coideal subalgebra contained in $B$. 
\end{rema}
\begin{rema}
  Let $B$ be a right coideal subalgebra of a quasitriangular Hopf algebra $H$. Let $\cA$ be the category of $H$-modules and let $\cB$ be the category obtained by restricting modules in $\cA$ to $B$ and allowing $B$-module homomorphisms. Then the pair $(\cB,\cA)$ is a tensor pair as defined in \cite{a-tD98}, see also \cite[4.1]{a-BalaKolb15p}. Assume additionally that $B$ is a quasitriangular coideal subalgebra. In this case relations (K$1'$) and (K$3'$) imply that multiplication by $\cK$ defines a cylinder braiding for the tensor pair $(\cB,\cA)$ as defined in \cite{a-tD98}, \cite[Section 4]{a-tDHO98}. Note that axiom (K$2'$) is not required to obtain the cylinder braiding. 
\end{rema}
\section{The universal K-matrix for quantum symmetric pairs}\label{sec:univ-K-qsp}
We now turn to the theory of quantum symmetric pairs. In this section we show that the category of finite dimensional representations of a quantum symmetric pair coideal subalgebra is a braided module category. In Sections \ref{sec:qg-notation} and \ref{sec:equivariantization} we recall quantized enveloping algebras and discuss the $\Z_2$-equivariantization of their category of finite dimensional representations. In Section \ref{sec:qsp} we fix notation for quantum symmetric pairs and recall essential results from \cite{a-BalaKolb15p}. These results are subsequently developed further to obtain Theorem \ref{thm:scrK} which provides all properties of a universal K-matrix as in Definition \ref{def:K-matrix} in the setting of quantum symmetric pairs. 
\subsection{Quantum group notation}\label{sec:qg-notation}
Let $\gfrak$ be a finite dimensional complex semisimple Lie algebra with a fixed Cartan subalgebra $\hfrak$ and a fixed set of simple roots $\roots=\{\alpha_i\,|\,i\in I\}$. Here $I$ denotes the set of nodes of the corresponding Dynkin diagram. Let $Q$ denote the root lattice of $\gfrak$, write $Q^+=\oplus_{i\in I}\N_0\alpha_i$, and let $P$ denote the weight lattice. We write $W$ to denote the Weyl group generated by the simple reflections $\sigma_i$ for $i\in I$. Let $(\cdot,\cdot)$ denote the invariant scalar product on the real vector space generated by $\Pi$ such that $(\alpha,\alpha)=2$ for all short roots $\alpha$ in each component.

Fix a field $\field$ of characteristic zero and let $q$ denote an indeterminate. Choose $d\in \N$ minimal such that $(\mu,\nu)\in \frac{1}{d}\Z$ for all $\mu,\nu\in P$. The quantized enveloping algebra $\uq=\uqg$ is the $\field(q^{1/d})$-algebra generated by elements $E_i, F_i, K_i$, and $K_i^{-1}$ for all $i\in I$ and relations given in \cite[3.1.1]{b-Lusztig94}. It is a Hopf algebra with coproduct $\kow$, counit $\vep$, and antipode $S$ given by
\begin{align}
  \kow(E_i)&=E_i \ot 1 + K_i\ot E_i,& \vep(E_i)&=0, & S(E_i)&=-K_i^{-1} E_i,\label{eq:E-copr}\\
  \kow(F_i)&=F_i\ot K_i^{-1} + 1 \ot F_i,&\vep(F_i)&=0, & S(F_i)&=-F_i K_i,\label{eq:F-copr}\\
  \kow(K_i)&=K_i \ot K_i,& \vep(K_i)&=1, & S(K_i)&=K_i^{-1}.\label{eq:K-copr}
\end{align}
We write $U^+$, $U^-$, and $U^0$ to denote the subalgebras generated by the elements of the sets $\{E_i\,|\,i\in I\}$, $\{F_i\,|\,i\in I\}$, and $\{K_i,K_i^{-1}\,|\,i \in I\}$ respectively. For any $U^0$-module $V$ and any $\lambda\in P$ let
\begin{align*}
  V_\lambda=\{v\in V\,|\,K_i v = q^{(\lambda,\alpha_i)}v \mbox{ for all $i\in I$}\}
\end{align*}
denote the corresponding weight space. We will apply this notation in particular to $U^+$ and $U^-$ which are $U^0$-modules with respect to the left adjoint action.
For $i\in I$ the commutator of $F_i$ with elements in $U^+$ can be expressed in terms of two skew derivations $r_i, {}_ir: U^+ \rightarrow U^+$ which are uniquely determined by the following property
\begin{align}\label{eq:riir}
  [x,F_i]= \frac{1}{(q_i-q_i^{-1})}\big(r_i(x)K_i -  K_i^{-1}{}_ir(x)\big) \qquad \mbox{for all $x\in U^+$,}
\end{align}
see \cite[Proposition 3.1.6]{b-Lusztig94}.

It is useful to work with completions $\sU$ and $\sU^{(2)}_0$ of $\uq$ and $\uq\ot\uq$, respectively, which have been discussed in detail in \cite[Section 3]{a-BalaKolb15p}. In brief, let $\Oint$ denote the category of finite dimensional $\uq$-modules of type 1, and let $\For:\Oint\rightarrow {\mathcal V}ect$ be the forgetful functor into the category of $\field(q^{1/d})$-vector spaces. The set of natural transformations $\sU=\End(\For)$ has the structure of a $\field(q^{1/d})$-algebra which contains $\uq$.
In this paper, as in \cite{a-BalaKolb15p}, we use the notion of a completion somewhat laxly to denote an overalgebra, and we ignore any questions of topology.

Important elements in $\sU$ are the Lusztig automorphisms $T_i$ for $i\in I$, defined as $T''_{i,1}$ in \cite[5.2.1]{b-Lusztig94}, which provide a braid group action on any $V\in Ob(\Oint)$. For any $w\in W$ with reduced expression $w=\sigma_{i_1} \dots \sigma_{i_t}$ we define $T_w=T_{i_1}\dots T_{i_t}\in \sU$. We also use the symbols $T_i$ for $i\in I$ and $T_w$ for $w\in W$ to denote the corresponding algebra automorphisms of $\uq$.

For any weight $\mu\in P$ define $K_\mu\in \sU$ by $K_\mu v_\lambda = q^{(\mu,\lambda)} v_\lambda$ for all weight vectors $v_\lambda$ of weight $\lambda$. Let $\Uc^0$ denote the subalgebra of $\sU$ generated by $\{K_\mu\,|\,\mu\in P\}$. The subalgebra $\Uc$ of $\sU$ generated by $U$ and $\Uc^0$ is a Hopf algebra which is often called the simply connected form of the quantized enveloping algebra. 

More generally, for any $n\in \N$ consider the functor
\begin{align}\label{eq:Un-def}
  \For^{(n)}:\underbrace{\Oint\times \dots \times \Oint}_{\text{$n$-factors}} \rightarrow {\mathcal V}ect, \qquad (V_1,\dots,V_n)\mapsto V_1\ot \cdots \ot V_n.
\end{align}
The set $\sU^{(n)}_0=\End(\For^{(n)})$ is a $\field(q^{1/d})$-algebra which contains $\uq^{\ot n}$. Observe that $\sU^{(1)}_0=\sU$. The multiplication in $\sU^{(n)}_0$ is given by composition of natural transformations. For $n\ge 2$ we denote the multiplication in $\sU^{(n)}_0$ by $\cdot$ for better readability. The coproduct $\kow:\uq\rightarrow \uq\ot \uq$ can be extended to an algebra homomorphism $\kow: \sU \rightarrow \sU^{(2)}_0$. Similarly, we can extend $\id \ot \kow$ and $\kow\ot \id: \uq^{\ot 2}\rightarrow \uq^{\ot 3}$ to maps $\id \ot \kow$ and $\kow \ot \id:\sU^{(2)}_0 \rightarrow \sU^{(3)}_0$.
\begin{eg}\label{eg:kappa}
  Let $f:P\rightarrow P$ be any map. For every $V,W\in Ob(\Oint)$ define a linear map $\kappa_{V,W}^f: V\ot W\rightarrow V\ot W$ by the property that
  \begin{align*}
     \kappa_{V,W}^f(v_\mu\ot w_\lambda)= q^{(f(\mu),\lambda)}v_\mu\ot w_\lambda \qquad \mbox{for all $v_\mu\in V_\mu$, $w_\lambda\in W_\lambda$.}
  \end{align*}
The collection $\kappa^f=(\kappa_{V,W}^f)_{V,W\in Ob(\Oint)}$ defines an element in $\sU^{(2)}_0$. We write $\kappa=\kappa^\id$ and $\kappa^-=\kappa^{-\id}$.  
\end{eg}
Another important element in $\sU^{(2)}_0$ is the quasi R-matrix $R$ for $\uq$. Here we recall the characterization of $R$ in terms of the bar involution for $\uq$, see \cite[Section 4.1]{b-Lusztig94}.
The bar involution for $\uq$ is the $\field$-algebra automorphism $\barU:\uq\rightarrow \uq$ defined by
\begin{align}
  \overline{q^{1/d}}^U&=q^{-1/d}, & \overline{E_i}^U&=E_i, & \overline{F_i}^U&=F_i,& \overline{K_i}^U&=K_i^{-1}.\label{eq:defbarUqg}
\end{align}
The bar involution $\barU$ is extended to $\uq\ot \uq$ diagonally as $\barU\ot \barU$. This extension does not map the diagonal subalgebra $\kow(\uq)$ to itself. However, there exists an element $R\in \prod_{\mu\in Q^+} U^-_\mu\ot U^+_\mu\subset \sU^{(2)}_0$, the quasi R-matrix, which satisfies the intertwiner relation
\begin{align}\label{eq:kowR=Rkow}
  \kow(\overline{u}^U) \cdot R= R \cdot (\barU\ot\barU)(\kow(u)) \qquad \mbox{for all $u\in \uq$.}
\end{align}
By slight abuse of notation we write $R=\sum_{\mu\in Q^+} R_\mu$  with $R_\mu\in U^-_\mu\ot U^+_\mu$. The quasi R-matrix  $R$ is uniquely determined by the intertwiner relation \eqref{eq:kowR=Rkow} together with the normalization $R_0=1\ot 1$, see \cite[Theorem 4.1.2]{b-Lusztig94}. 

Recall the element $\kappa^-\in \sU^{(2)}_0$ defined in Example \ref{eg:kappa}. The element 
\begin{align*}
  \Rmat=R_{21} \cdot \kappa^- \in \sU^{(2)}_0
\end{align*}
is a universal R-matrix for $\uq$. In other words, the element $\Rmat$ satisfies
\begin{align}\label{eq:RuuopR}
  \Rmat\cdot \kow(u)&=\kow^\op(u)\cdot \Rmat \qquad \mbox{for all $u\in \sU$}
\end{align}
in $\sU^{(2)}_0$, and Properties (2) and (3) of Definition \ref{def:universalR} hold as relations in $\sU^{(3)}_0$, see \cite[Proposition 32.2.4]{b-Lusztig94}.

Recall from \cite[Definition XIV.6.1]{b-Kassel1} that a ribbon algebra $(H,\kow,\vep,S,\Rmat,\rib)$ is a quasitriangular Hopf algebra $(H,\kow,\vep,S,\Rmat)$ with a central element $\rib\in H$ such that
\begin{align}\label{eq:rib-conditions}
  \kow(\rib) &= (\Rmat_{21}\Rmat)^{-1}\rib\ot \rib, & \vep(\rib)&=1, & S(\rib)=\rib.
\end{align} 
The algebra $\sU$ contains a ribbon element. Let $P^+\subset P$ be the set of dominant integral weights. For any $\lambda\in P^+$ let $V(\lambda)$ denote the corresponding irreducible $\uq$-module of highest weight $\lambda$. As $\Oint$ is a semisimple category with simple objects $V(\lambda)$, $\lambda\in P^+$, we can define an element $\rib=(\rib_V)_{V\in Ob(\Oint)}\in \sU$ by 
\begin{align}\label{eq:rib}
  \rib_{V(\lambda)}=q^{(\lambda,\lambda+2\rho)}\id_{V(\lambda)}.
\end{align}
By construction $\rib$ is a central element in $\sU$. In \cite{a-Drin90-almost} Drinfeld showed that $\rib$ satisfies the first relation in \eqref{eq:rib-conditions}. The other relations in \eqref{eq:rib-conditions} also hold in $\sU$, and hence $\rib$ is a ribbon element for $\uq$ in a completed sense. This is made more precise in terms of the category $\Oint$. Let $\mathbbm{1}\in Ob(\Oint)$ denote the trivial representation of $\uq$. Recall the notion of a ribbon category from \cite[1.4]{b-Turaev}. The following well-known proposition is a consequence of the properties of the universal R-matrix $\Rmat$ and the twist $v$.
\begin{prop}\label{prop:braidedOint}
  The category $(\Oint,\ot, \mathbbm{1})$ is a braided monoidal category with braiding given by
\begin{align}\label{eq:braiding}
  c=\{c_{V,W}=\flip \circ \Rmat:V\ot W \rightarrow W\ot V\}_{V,W\in Ob(\Oint)}.
\end{align}
Moreover, $(\Oint,\ot, \mathbbm{1}, c, v^{-1})$ has the structure of a ribbon category with twist given by the inverse of the element $\rib\in \sU$ defined by \eqref{eq:rib}.
\end{prop}
\subsection{Involutive diagram automorphism and equivariantization}\label{sec:equivariantization}
Let $\sigma:I\rightarrow I$  be a diagram automorphism such that $\sigma^2=\id$. We allow the case $\sigma=\id$. The map $\sigma$ induces an involutive Hopf algebra automorphism $\sigma_U:U\rightarrow U$ defined by $\sigma_U(X_i)=X_{\sigma(i)}$ for $X=E,F,K^{\pm 1}$. Let $\langle \sigma \rangle$ be the cyclic group generated by $\sigma$. Consider the semidirect product $U_\sigma=U\rtimes \field(q^{1/d}) \langle \sigma \rangle$ where the commutation relation is given by
\begin{align*}
  \sigma u = \sigma_U(u) \sigma \qquad \mbox{for all $u\in U$.}
\end{align*}
The algebra $U_\sigma$ is a Hopf algebra with $\kow(\sigma)=\sigma\ot \sigma$.
For any $V\in Ob(\Oint)$ consider the twisted representation $V^\sigma\in Ob(\Oint)$ which coincides with $V$ as a vector space with the action $\lact_\sigma$ given by
\begin{align*}
  u\lact_\sigma v = \sigma_U(u)v \qquad \mbox{for all $u\in U$, $v\in V$.}
\end{align*}
If $\sigma \neq \id$, finite dimensional representations of $U_\sigma$ which are in $\Oint$ when restricted to $U$ can be written as pairs $(V,f)$ where $V\in Ob(\Oint)$ and $f:V\rightarrow V^\sigma$ is a $U$-module isomorphism such that $f^2=\id$. A $U_\sigma$-module homomorphism $(V,f)\rightarrow (W,g)$ consists of a $U$-module homomorphism $\varphi:V\rightarrow W$ such that $\varphi \circ g = f\circ \varphi$. The category $\Oint^{\langle \sigma\rangle}$ of $U_\sigma$-modules which belong to $\Oint$ when restricted to $U$ is hence nothing but the $\langle \sigma\rangle$-equivariantization of $\Oint$, see \cite[Section 4]{a-DGNO10}. This remains true if $\sigma=\id$. Note that in this case $\Oint^{\langle \sigma\rangle}=\Oint$ has objects $(V,\id)$. The category $\Oint^{\langle\sigma \rangle}$ retains all structure from $\Oint$. Indeed, tensor products are defined by $(V,f)\ot(W,g)=(W\ot V,f\ot g)$ and the trivial representation $\mathbbm{1}$ extends to $U_\sigma$. The braiding defined by \eqref{eq:braiding} and the twist $v^{-1}$ defined by \eqref{eq:rib} also survive under equivariantization.
\begin{cor}\label{cor:Os-rib}
  The datum $(\Oint^{\langle\sigma\rangle},\ot,\mathbbm{1},c,v^{-1})$ has the structure of a ribbon category. 
\end{cor}  
If $\sigma \neq \id$ then there are three types of irreducible $U_\sigma$-modules in $\Oint^{\langle\sigma\rangle}$. Let $\lambda\in P^+$ with $\sigma(\lambda)=\lambda$ and let $v_\lambda\in V(\lambda)$ be a highest weight vector. Then the $\Uq$-module structure on $V(\lambda)$ can be extended to a $U_\sigma$-module structure by $\sigma(v_\lambda)=v_\lambda$ or $\sigma(v_\lambda)=-v_\lambda$. We denote the corresponding irreducible $U_\sigma$-modules by $V(\lambda)_+$ and $V(\lambda)_-$, respectively. If $\mu\in P^+$ and $\sigma(\mu)\neq \mu$ then $W(\mu)=V(\mu)\oplus V(\mu)^\sigma\cong V(\mu)\oplus V(\sigma(\mu))$ has the structure of an irreducible $U_\sigma$ module with $\sigma((v,w))=(w,v)$. Observe that $W(\mu)\cong\mathrm{Ind}_U^{U_\sigma}V(\mu)$ if $\sigma(\mu)\neq \mu$ while $\mathrm{Ind}_U^{U_\sigma}V(\lambda)\cong V(\lambda)_+ \oplus V(\lambda)_-$ if $\sigma(\lambda)=\lambda$. To avoid double counting, let $P^+_\sigma$ denote a set of representatives of the $\sigma$-orbits in $P^+$. Observe that $W(\sigma(\mu))\cong W(\mu)$ if $\sigma(\mu)\neq \mu$. We summarize the structure of the equivariantization $\Oint^{\langle \sigma \rangle}$ in the following lemma.
\begin{lem}\label{lem:Os-ss}
  The category $\Oint^{\langle \sigma \rangle}$ is semisimple. If $\sigma \neq \id$ then any irreducible object in $\Oint^{\langle \sigma\rangle}$ is isomorphic to one of $V(\lambda)_+, V(\lambda)_-$ or $W(\mu)$ for some $\lambda, \mu\in P^+_\sigma$ with $\sigma(\lambda)=\lambda$ and $\sigma(\mu)\neq \mu$.
\end{lem}
\begin{proof}
  Assume that $\sigma\neq \id$ and $V\in Ob(\Oint^{\langle \sigma\rangle})$. As a $U$-module $V$ decomposes into a direct sum
  \begin{align*}
     V\cong \bigoplus_{\lambda\in P^+_\sigma  \atop \sigma(\lambda)=\lambda} V(\lambda)^{a_\lambda} \bigoplus_{\mu\in P^+_\sigma\atop \sigma(\mu)\neq \mu} \big(V(\mu)\oplus V(\sigma(\mu))\big)^{a_\mu}
  \end{align*}
  for some $a_\lambda,a_\mu\in \N_0$. For any $\lambda\in P^+$ with $\sigma(\lambda)=\lambda$ the weight space of highest weight $\lambda$ in $V(\lambda)^{a_\lambda}$ decomposes into $\pm 1$ eigenspaces under the action of $\sigma$. This shows that the component $V(\lambda)^{a_\lambda}$ is isomorphic to a direct sum of copies of $V(\lambda)_+$ and $V(\lambda)_-$. Similarly, for $\mu\in P^+_\sigma$ with $\sigma(\mu)\neq \mu$ one gets $\sigma(V(\mu)^{a_\mu})\cong V(\sigma(\mu))^{a_\mu}$. This implies that the component $\big(V(\mu)\oplus V(\sigma(\mu))\big)^{a_\mu}$ is isomorphic to $W(\mu)^{a_\mu}$ as a $U_\sigma$-module.
\end{proof}  
\begin{rema}\label{rem:Vlp}
  To unify notation we set $V(\lambda)_+=V(\lambda)$ in the case $\sigma=\id$.
\end{rema}  
We also extend the completion $\sU$ to the equivariantization. Let $\cF or_\sigma:\Oint^{\langle \sigma\rangle}\rightarrow \cV ect$ be the forgetful functor and define $\sU_{\sigma}=\End(\cF or_\sigma)$. By construction $\sU_{\sigma}$ contains the element $(f:V\rightarrow V\,|\,(V,f)\in Ob(\Oint^{\langle \sigma\rangle}))$ which by abuse of notation we again denote by $\sigma$. Moreover, $\sU$ and $\Uc_\sigma=\Uc\rtimes \field(q^{1/d})\langle\sigma \rangle$ are subalgebras of $\sU_{\sigma}$.

In the following we will often write objects in $\Oint^{\langle\sigma\rangle}$ as $V$ and imply that the map $f:V\rightarrow V^\sigma$ is understood.
Similarly to \eqref{eq:Un-def}, for any $n\in\N$ we consider the functor
\begin{align}\label{eq:Uns-def}
  \For^{(n)}_\sigma:\underbrace{\Oint^{\langle \sigma\rangle}\times \dots \times \Oint^{\langle \sigma\rangle}}_{\text{$n$-factors}} \rightarrow {\mathcal V}ect, \qquad (V_1,\dots,V_n)\mapsto V_1\ot \cdots \ot V_n.
\end{align}
and set $\sU^{(n)}_{0,\sigma}=\End(\For^{(n)}_\sigma)$. Again we have an inclusion $\sU^{(n)}_0\subset \sU^{(n)}_{0,\sigma}$. As in Section \ref{sec:qg-notation}, the coproduct extends to an algebra homomorphism $\kow:\sU_{\sigma}\rightarrow \sU_{0,\sigma}^{(2)}$.
\subsection{Quantum symmetric pairs}\label{sec:qsp}
Quantum symmetric pair coideal subalgebras of $\uq=\uqg$ were introduced by G.~Letzter in \cite{a-Letzter99a}. Here we briefly recall their construction following \cite{a-Kolb14}. See also \cite[Section 5]{a-BalaKolb15p} for a slightly more detailed overview in the conventions of the present paper. Let $X\subset I$ be a proper subset and let $\tau:I\rightarrow I$ be an involutive diagram automorphism with $\tau(X)=X$. Assume that $(X,\tau)$ satisfies the properties of an admissible pair (or Satake diagram) given in \cite[Definition 2.3]{a-Kolb14}. Let $\theta=\theta(X,\tau):\gfrak\rightarrow \gfrak$ denote the corresponding involutive Lie algebra automorphism defined in \cite[(2.8)]{a-Kolb14}, \cite[Theorem 5.3]{a-BalaKolb15p}, and let $\kfrak=\{x\in \gfrak\,|\,\theta(x)=x\}$ denote the corresponding fixed Lie subalgebra. The involution $\theta$ leaves the Cartan subalgebra $\hfrak$ invariant. The dual map $\Theta:\hfrak^\ast\rightarrow \hfrak^\ast$ is given by
\begin{align}\label{eq:Theta}
  \Theta=-w_X \circ \tau
\end{align}
where $w_X$ denotes the longest element in the parabolic subgroup $W_X$ of the Weyl group $W$. The involutive automorphism $\theta$ can be deformed to an algebra automorphism $\theta_q:\uq\rightarrow \uq$, see \cite[Definition 4.3]{a-Kolb14} for details. 

Let $\cM_X=U_q(\gfrak_X)$ denote the Hopf subalgebra of $\uq$ generated by the elements $E_i, F_i, K_i$, and $K_i^{-1}$ for all $i\in X$. Let $\Uc^0_\Theta$ denote the subalgebra of $\Uc^0$ generated by all $K_\mu$ with $\mu\in P$ and $\Theta(\mu)=\mu$.

Quantum symmetric pair coideal subalgebras depend on a choice of parameters $\bc=(c_i)_{i\in I\setminus X}$ and $\bs=(s_i)_{i\in I\setminus X}$ with $c_i\in \field(q^{1/d})^\times$ and $s_i\in \field(q^{1/d})$. These parameters are subject to certain restrictions. More specifically, one has $\bc\in \cC$ with
\begin{align*}
  \cC=\{ \mathbf{c}\in (\mathbb{K}(q^{1/d})^\times)^{I\setminus X} | c_i=c_{\tau(i)} \textrm{ if } \tau(i)\ne i \textrm{ and } (\alpha_i,\Theta(\alpha_i))=0 \}, 
\end{align*}
and $\bs\in \cS$ for a parameter set $\cS$ which is defined in \cite[(5.11)]{a-Kolb14}, see also \cite[5.7]{a-BalaKolb15p}, \cite[Remark 3.3]{a-BalaKolb15}.
With these notations we are ready to define quantum symmetric pair coideal subalgebras inside the simply connected form of quantized enveloping algebras. 
\begin{defi}
  \label{def:qsp}
Let $(X,\tau)$ be an admissible pair, $\bc=(c_i)_{i\in I\setminus X}\in \cC$, and $\bs=(s_i)_{i\in I\setminus X}\in \cS$. The quantum symmetric pair coideal subalgebra $\coid=\coid(X,\tau)$ is the subalgebra of $\Uc$ generated by $\mathcal{M}_X$, $\Uc^0_\Theta{}$, and the elements 
\begin{align}
  B_i &= F_i + c_i \theta_q(F_iK_i) K_i^{-1} + s_i K_i ^{-1} \label{eq:Bi-def} 
\end{align}
for all $i \in I\setminus X$.
\end{defi}
Let $\rho_X$ denote the half sum of positive roots of the semisimple Lie algebra $\gfrak_X$.
All through this paper we assume that the parameters $\bc=(c_i)_{i\in I\setminus X}$ satisfy the condition
\begin{align}\label{eq:c-bar-condition}
  c_{\tau(i)}= q^{(\alpha_i,\Theta(\alpha_i)-2\rho_X)}\overline{c_i}^U  \qquad \mbox{for all $i\in I\setminus X$.}
\end{align}
In this case it was proved in \cite[Theorem 3.11]{a-BalaKolb15} that there exists an involutive $\field$-algebra automorphism $\barB:B_{\bc,\bs}\rightarrow B_{\bc,\bs}$, $x\mapsto \overline{x}^B$ such that
\begin{align}
		\overline{x}^B=\overline{x}^U \textrm{ for all }  x \in \mathcal{M}_X \Uc^0_\Theta{}, \qquad  	  
		\overline{B_i}^B&=B_i \textrm{ for all } i \in I\setminus X. \label{eq:defbarBcs}
\end{align} 
In particular, $\overline{q^{1/d}}^B=q^{-1/d}$. In a similar way as the quasi R-matrix provides an intertwiner between two bar-involutions on $\kow(\uq)$, there exists an intertwiner for the two involutions $\barU$ and $\barB$ on $\Bcs$. More explicitly, assume that the parameters $\bs=(s_i)_{i\in I\setminus X}$ satisfy
\begin{align}\label{eq:s-X-condition}
  \overline{s_i}^U=s_i \qquad \mbox{for all $i\in I\setminus X$.}
\end{align}
In this case it was proved in \cite[Theorem 6.10]{a-BalaKolb15p} that there exists a uniquely determined element $\Xfrak=\sum_{\mu\in Q^+}\Xfrak_\mu\in \prod_{\mu\in Q^+} U^+_\mu$ with $\Xfrak_0=1$ and $\Xfrak_\mu\in U^+_\mu$ such that the equality
\begin{align}\label{eq:Xintertwine}
  \overline{x}^B\,\Xfrak = \Xfrak\, \overline{x}^U
\end{align}
holds in $\sU$ for all $x\in \Bcs$. For quantum symmetric pairs of type AIII/IV with $X=\emptyset$, the existence of $\Xfrak$ was first noted by H.~Bao and W.~Wang in \cite{a-BaoWang13p}. The element $\Xfrak$ is called the 1-tensor quasi K-matrix for the quantum symmetric pair coideal subalgebra $\Bcs$.

Following a program outlined in \cite{a-BaoWang13p}, the quasi K-matrix $\Xfrak$ was used in \cite{a-BalaKolb15p} to construct an element $\cK\in \sU$ which satisfies properties (K$1'$) and (K$3'$) from Section \ref{sec:quasi-coid} up to a twist and completion. To fix notation, we briefly recall the construction of $\cK$. We work with a suitable equivariantization of $\Oint$ to avoid the twist. Choose a group homomorphism $\gamma:P\rightarrow \field(q^{1/d})^\times$ such that $\gamma(\alpha_i)=c_i s(\tau(i))$ for all $i\in I\setminus X$ and $\gamma(\alpha_i)=0$ if $i\in X$, see \cite[Section 8.4]{a-BalaKolb15p}. Here $s(\tau(i))\in \field^\times$ is a constant which enters the construction of the involutive automorphism $\theta(X,\tau)$ of $\gfrak$, see \cite[Section 5.1]{a-BalaKolb15p}. For any $\lambda \in P$ write $\lambda^+=(\lambda+\Theta(\lambda))/2$ and $\lambdatil=(\lambda-\Theta(\lambda))/2$. With this notation define a function $\xi:P\rightarrow \field(q^{1/d})^\times$ by
\begin{align*}
  \xi(\lambda) = \gamma(\lambda) q^{-(\lambda^+,\lambda^+) + \sum_{k\in I} (\alphatil_k,\alphatil_k) \lambda(\varpi_k^\vee)}
\end{align*}
where $\varpi_i^\vee$ for $i\in I$ denote the fundamental coweights, see also \cite[Remark 8.1]{a-BalaKolb15p}. By \cite[Lemma 8.2]{a-BalaKolb15p}, the function $\xi$ satisfies the relation
\begin{align}\label{eq:ximunu}
  \xi(\mu+\nu) = \xi(\mu) \xi(\nu) q^{-(\mu+\Theta(\mu),\nu)} \qquad \mbox{for all $\mu,\nu\in P$.}
\end{align}
We consider the map $\xi$ as an element in $\sU$, see \cite[Example 3.3]{a-BalaKolb15p}. Then \eqref{eq:ximunu} implies that
\begin{align}\label{eq:kow-xi}
  \kow(\xi) = \xi\ot \xi \cdot \kappa^- \cdot \kappa^{-\Theta}.
\end{align}
Let $w_0$ denote the longest element in the Weyl group $W$. Define a diagram automorphism $\tau_0:I\rightarrow I$ by the property that $w_0(\alpha_i)=-\alpha_{\tau_0(i)}$ for all $i\in I$. We set $\sigma=\tau\tau_0$.
In addition to \eqref{eq:c-bar-condition} we need to make the assumption that the parameters ${\bf c}\in \cC$ satisfy the relation
\begin{align}\label{eq:c-tau-condition}
  c_{\sigma(i)}=c_{i} \qquad \mbox{for all $i\in I\setminus X$},
\end{align}
see \cite[7.1 and Remark 7.2]{a-BalaKolb15p}. In this case $\sigma_U:U\rightarrow U$ restricts to an involutive algebra automorphism of $B_\bc$. Now consider the algebra $\sU_\sigma$ defined in Section \ref{sec:equivariantization}. Recall that $w_X$ denotes the longest element in the parabolic subgroup $W_X$. The 1-tensor universal K-matrix for $\Bcs$ is defined by
\begin{align}\label{eq:K-def}
  \cK= \Xfrak \xi T_{w_X}^{-1} T_{w_0}^{-1} \sigma \in \sU_\sigma.
\end{align}
The following theorem summarizes the main result of \cite{a-BalaKolb15p} in the case where $\gfrak$ is semisimple.
\begin{thm}[{\upshape \cite[Corollary 7.7, Theorem 9.5]{a-BalaKolb15p}}]\label{thm:BK}
  The element $\cK\in \sU_\sigma$ defined by \eqref{eq:K-def} has the following properties:
  \begin{align}
    \cK \,b &= b\,\cK \qquad \mbox{for all $b\in \Bcs$,}\label{eq:KbttbK}\\
    \kow(\cK) &=  \Rmat_{21}\cdot (1\ot\cK)\cdot \Rmat \cdot(\cK\ot 1). \label{eq:kowK-qsp}
  \end{align}
\end{thm}
Using relation  \eqref{eq:RuuopR} and the fact that $\sigma\ot \sigma(\Rmat)=\Rmat$ one obtains 
\begin{align}\label{eq:kowK-qsp2}
  \kow(\cK) &=  (\cK\ot 1)\cdot\Rmat_{21}\cdot (1\ot\cK)\cdot \Rmat
\end{align}
which provides the reflection equation by equating the right hand sides of \eqref{eq:kowK-qsp} and \eqref{eq:kowK-qsp2}.
The above theorem shows that the 1-tensor universal K-matrix $\cK$ satisfies properties (K$1'$) and (K$3'$) from Section \ref{sec:quasi-coid}. In the remainder of Section \ref{sec:univ-K-qsp} we show that $\cK$ also satisfies property (K$2'$).
\begin{rema}
  In \cite[Corollary 7.7]{a-BalaKolb15p} the 1-tensor universal K-matrix was defined by $\cK=\Xfrak \xi T_{w_X}^{-1} T_{w_0}^{-1}$ omitting the factor $\sigma$ in \eqref{eq:K-def}. This lead to twisted versions of \eqref{eq:KbttbK} and \eqref{eq:kowK-qsp} and a twisted reflection equation which involved the automorphism $\sigma$. Working in $\sU_\sigma$ and defining $\cK$ by \eqref{eq:K-def} we obtain the untwisted relations as in Section \ref{sec:BMC}. 
\end{rema}  
\subsection{Definition of $R^\theta$ and first properties}
As a preparation we first consider an element $R^\theta$ which was introduced in the special case of quantum symmetric pairs of type AIII/IV with $X=\emptyset$ in \cite[Section 3]{a-BaoWang13p}. Define
\begin{align}\label{eq:Rtheta-def}
  R^\theta=\kow(\Xfrak) \cdot R \cdot(\Xfrak^{-1} \ot 1)\in \sU^{(2)}_0.
\end{align}
In \cite{a-BaoWang13p} the element $R^\theta$ is called the \textit{quasi R-matrix} for $\Bcs$. This terminology  stems from the following intertwiner property which is similar to the intertwiner property of the quasi R-matrix for $\uq$ in Equation \eqref{eq:kowR=Rkow}. Here we prefer to call $R^\theta$ the 2-tensor quasi K-matrix for $\Bcs$.
\begin{prop}[{\upshape \cite[Proposition 3.2]{a-BaoWang13p}}] \label{prop:quasiRB}
For any $b\in \Bcs$ one has
  \begin{align*}
    \kow(\overline{b}^B) \cdot R^\theta = R^\theta \cdot (\ophan^B\ot \ophan^U)\circ \kow(b)
  \end{align*}
  in $\sU^{(2)}_0$.
\end{prop}
\begin{proof}
   Let $b\in \Bcs$. Using the intertwiner relations \eqref{eq:kowR=Rkow} and \eqref{eq:Xintertwine} one calculates
  \begin{align*}
    R^\theta \cdot (\ophan^B\ot \ophan^U)\circ \kow(b) 
    &= \kow(\Xfrak) \cdot R \cdot (\Xfrak^{-1}\ot 1)\cdot (\ophan^B\ot \ophan^U)\circ \kow(b)\\
    &= \kow(\Xfrak) \cdot R \cdot(\ophan^U\ot \ophan^U)\circ \kow(b)\cdot(\Xfrak^{-1}\ot 1)\\
    &=\kow(\Xfrak)\cdot \kow(\overline{b}^U)\cdot R \cdot(\Xfrak^{-1}\ot 1)\\
    &=\kow(\overline{b}^B) \cdot \kow{\Xfrak}\cdot R \cdot (\Xfrak^{-1}\ot 1)
  \end{align*}
  which proves the proposition.
\end{proof}
Observe that $\widehat{\sU\ot U^+}=\prod_{\mu\in Q^+} \sU \otimes U^+_\mu$
is a subalgebra of $\sU^{(2)}_0$. As $\kow(\Xfrak)$, $R$ and $\Xfrak^{-1}\ot 1$ belong to $\widehat{\sU\ot U^+}$ we obtain $R^\theta\in\widehat{\sU\ot U^+}$ from \eqref{eq:Rtheta-def}. Hence we can write formally
\begin{align}\label{eq:Rtheta-sum}
    R^\theta= \sum_{\mu\in Q^+} R^\theta_\mu \qquad \mbox{with $R_\mu^\theta\in \sU\otimes U^+_\mu$.}
  \end{align}
The situation is even better. One has the following result which is a generalization of \cite[Proposition 3.5]{a-BaoWang13p}.
\begin{prop}\label{prop:Rthetamu}
  For any $\mu\in Q^+$ one has $R^\theta_\mu \in B_{\bc,\bs}\otimes U^+_\mu$.
\end{prop}
\begin{proof}
  For any $i\in I$ one has
  \begin{align*}
    \kow(B_i) = B_i \ot K_i^{-1} + 1 \ot F_i + c_i M_i
  \end{align*}
  for some $M_i\in \cM_X \Uc^0_\Theta \ot \sum_{\gamma\ge \alpha_{\tau(i)}} U^+_\gamma K_i^{-1}$, see \cite[Proof of Proposition 5.2]{a-Kolb14}. Hence
  Proposition \ref{prop:quasiRB} implies that
  \begin{align*}
    \big(B_i \ot K_i^{-1} + 1 \ot F_i + c_i M_i\big)\cdot R^\theta = R^\theta \cdot
    \big(B_i \ot K_i + 1 \ot F_i + \overline{c_i} (\ophan^B\ot \ophan^U)(M_i)\big).
  \end{align*}
  Hence property \eqref{eq:riir} of Lusztig's skew derivations $r_i$ and ${}_ir$ leads to the relation
  \begin{align}
    (1\ot r_i) (R^\theta) &= - (q_i-q_i^{-1}) R^\theta \cdot \big(B_i\ot 1 + \overline{c_i} (\ophan^B\ot \ophan^U)(M_i)(1\ot K_i^{-1})\big) \label{eq:riRt}
  \end{align}
  which holds in $\sU^{(2)}_0$. We now show $R^\theta_\mu\in \Bcs\ot U^+_\mu$ by induction on $\hght(\mu)$. By definition of $R^\theta$ we have $R^\theta_0=1\ot 1 \in \Bcs\ot U^+_0$. Now assume that $\mu>0$ and $R^\theta_\nu\in \Bcs\ot U^+_\nu$ for all $\nu<\mu$. Then relation \eqref{eq:riRt} implies that
  \begin{align*}
      (1\ot r_i)(R^\theta_\mu) \in \Bcs\ot U^+_{\mu-\alpha_i}
  \end{align*}
  for all $i\in I$. Write
  \begin{align*}
    R^\theta_\mu = \sum_{j=1}^m a_j\ot u_j
  \end{align*}
  with linearly independent elements $u_j\in U^+_\mu$ and $a_j\in U$. Define a map
  \begin{align*}
    r:U^+_\mu \rightarrow \bigoplus_{i\in I} U^+_{\mu-\alpha_i}, \qquad u\mapsto \sum_{i\in I} r_i(u).
  \end{align*}
  By \cite[Lemma 1.2.15]{b-Lusztig94} the map $r$ is injective if $\mu>0$. Hence 
  \begin{align*}
    (1\ot r)(R^\theta_\mu)= \sum _{j=1}^m a_j\ot r(u_j)\in \Bcs\ot \bigoplus_{i\in I} U^+_{\mu-\alpha_i}
  \end{align*}
  has linearly independent second tensor factors. But this means that $a_j\in \Bcs$ for all $j=1,\dots,m$ and hence $R^\theta_\mu\in \Bcs\ot U^+_\mu$.
\end{proof}
\begin{rema}
  By \cite[Theorem 6.10]{a-BalaKolb15p} the 1-tensor quasi K-matrix $\Xfrak$ exists in the more general setting where $\gfrak$ is a symmetrizable Kac-Moody algebra and $(X,\tau)$ is an admissible pair in the sense of \cite[Definition 2.3]{a-Kolb14}.
  In this setting the 2-tensor quasi $K$-matrix $R^\theta\in \sU^{(2)}_0$ can still be defined by \eqref{eq:Rtheta-def}. Relation \eqref{eq:Rtheta-sum} still holds and the proof of the above proposition applies literally to show that $R^\theta_{\mu}\in \Bcs\ot U^+_\mu$ also in the Kac-Moody case. 
\end{rema}  
\subsection{Completion of $\Bcs\ot \uq^{\ot (n{-}1)}_\sigma$}\label{sec:Bcomplete}
To formulate a version of property (K$2'$) for the element $\cK\in \sU_\sigma$ defined by \eqref{eq:K-def} we need a completion of $\Bcs\ot \uq_\sigma$ inside  $\sU^{(2)}_{0,\sigma}$. More generally, we will define a completion of $\Bcs\ot \uq^{\ot (n{-}1)}_\sigma$ inside $\sU^{(n)}_{0,\sigma}$.

Let $\uq_\sigma^\ast$ denote the linear dual space of $\uq_\sigma$. For any $V\in Ob(\Oint^{\langle\sigma\rangle})$ and any $v\in V$, $f\in V^\ast$ the matrix coefficient $c_{f,v}\in \uq_\sigma^\ast$ is defined by $c_{f,v}(u)=f(uv)$ for all $u\in \uq_\sigma$. Let $\cA_\sigma$ denote the Hopf algebra generated by the matrix coefficients of all representations in $\Oint^{\langle \sigma\rangle}$. We write $\cA=\cA_\id$.
Let $n, m\in \N_0$ with $m\le n$. There exists a uniquely determined bilinear map
\begin{align}\label{eq:contraction}
  \langle\,,\,\rangle_{n,m}: \sU^{(n)}_{0,\sigma} \otimes \cA_\sigma^{\ot m} \rightarrow \sU^{(n-m)}_{0,\sigma}
\end{align}
such that
\begin{align*}
  \langle X,c_{f_1,v_1}\ot \dots&\ot c_{f_m,v_m}\rangle_{n,m}(w)\\ =&\big(\underbrace{\id \ot \dots \ot \id}_{\text{$(n{-}m)$ times}}\ot f_1\ot \dots\ot f_m\big)\big(X(w\ot v_1 \ot \dots \ot v_m)\big)
\end{align*}
for all $X\in \sU^{(n)}_{0,\sigma}$, $w\in W_1\ot \dots \ot W_{n-m}$, and $v_i\in V_i$, $f_i\in V_i^\ast$ where $V_i, W_j\in Ob(\Oint^{\langle\sigma \rangle})$ for $i=1,\dots, m$ and $j=1,\dots,n-m$. We call $\langle\,,\,\rangle_{n,m}$ the contraction pairing between $\sU^{(n)}_{0,\sigma}$ and $\cA^{\ot m}_\sigma$. We use Sweedler notation $\kow(a)=a_{(1)}\ot a_{(2)}$ for the tensor coalgebra structure on $\cA^{\ot m}_\sigma$. The contraction pairing respects the algebra structure on $\sU^{(n)}_{0,\sigma}$ in the sense that
\begin{align}\label{eq:hopf-contraction}
  \langle X Y,a\rangle_{n,m}= \langle X, a_{(1)}\rangle_{n,m} \langle Y, a_{(2)}\rangle_{n,m}  
\end{align}
for all $X,Y\in \sU^{(n)}_{0,\sigma}$ and all $a\in \cA^{\ot m}_\sigma$. 

With the help of the contraction pairing we are now ready to define the desired completion of $\Bcs\ot \uq^{\ot (n{-}1)}_\sigma$. For any $n\in \N$ with $n\ge 2$ define
\begin{align}\label{eq:BU-complete}
  \sB^{(n)}=\{X\in \sU^{(n)}_{0,\sigma}\,|\,\langle X,a\rangle_{n,n{-}1}\in \Bcs \mbox{ for all $a\in \cA_\sigma^{\ot (n{-}1)}$}\}
\end{align}
and set $\sB^{(1)}=\Bcs$.
Formula \eqref{eq:hopf-contraction} implies that $\sB^{(n)}$ is a subalgebra of $\sU^{(n)}_{0,\sigma}$ for any $n\in \N$. 
By definition $\Bcs\ot \sU_\sigma \subset \sB^{(2)}$, but the algebra $\sB^{(2)}$ is strictly bigger than $\Bcs\ot \sU_\sigma$. 
\begin{eg}\label{eg:kkT}
  Recall the notation $\kappa^f$ from Example \ref{eg:kappa} and the map $\Theta:P \rightarrow P$ defined by \eqref{eq:Theta}.
Consider the element $\kappa\cdot\kappa^{\Theta}\in \sU^{(2)}$. We claim that $\kappa\cdot\kappa^{\Theta}\in \sB^{(2)}$. Indeed, for any $V,W \in Ob(\Oint^{\langle \sigma\rangle})$ and any weight vector $v\in V$ of weight $\lambda$ one has
\begin{align*}
  \kappa \cdot \kappa^\Theta\big|_{W\ot v}=(K_{\lambda+\Theta(\lambda)}\ot \id) \big|_{W\ot v}.
\end{align*}
This implies that
\begin{align*}
  \langle \kappa \cdot \kappa^\Theta, c_{f,v}\rangle_{2,1}=f(v) K_{\lambda+\Theta(\lambda)}\in \Bcs  
\end{align*}
for any $f\in V^\ast$ and any weight vector $v\in V$ of weight $\lambda$. As any element of $\cA_\sigma$ is a linear combination of matrix coefficients $c_{f,v}$ for some weight vectors $v$, one obtains that $\kappa\cdot \kappa^\Theta\in \sB^{(2)}$.
\end{eg}  
\begin{eg}\label{eg:RthetainB2}
 The element $R^\theta$ defined by \eqref{eq:Rtheta-def} satisfies $R^\theta\in \sB^{(2)}$. This is a consequence of Proposition \ref{prop:Rthetamu}.
\end{eg}
Let $k,n\in \N$ with $1\le k\le n$. The application of the coproduct $\kow$ to the $k$-th tensor factor of $\uq^{\ot n}$ extends to an algebra homomorphism 
\begin{align}\label{eq:kowknU}
  \kow^{\sU,\sigma}_{n,k}:\sU^{(n)}_{0,\sigma}\longrightarrow \sU^{(n+1)}_{0,\sigma}
\end{align}
defined by
\begin{align*}
  \big(\kow^{\sU,\sigma}_{n,k}(\varphi)\big)_{V_1,V_2,\dots,V_{n+1}}= 
\varphi_{\raisebox{-1pt}{\tiny $V_{1},\dots,V_{k-1},V_{k}{\ot} V_{k+1},V_{k+2},\dots,V_{n+1}$}}.
\end{align*} 
The constructions imply that
\begin{align}\label{eq:kowB-inB2}
  \kow^{\sU,\sigma}_{n,k}(\sB^{(n)}) \subset \sB^{(n{+}1)}.
\end{align}
The maps $\kow^{\sU,\sigma}_{n,k}$ will be used in Section \ref{sec:braid-mod-cat}.

We end this section by considering the contraction pairing \eqref{eq:contraction} in the special case $n=2$, $m=1$. The following will only be used in Section \ref{sec:center}. For any  $X\in \sU^{(2)}_{0,\sigma}$ define a linear map 
\begin{align*}
  \ltil_X:\cA_\sigma \rightarrow \sU_\sigma, \qquad \ltil(a)=\langle X,a\rangle_{2,1}.
\end{align*}
We call $\ltil_X$ the generalized $l$-operator corresponding to the element $X\in \sU^{(2)}_{0,\sigma}$.  
\begin{rema}\label{rem:ltil}
  For $X=\Rmat\Rmat_{21}$ and $\sigma=\id$ the generalized $l$-operator $\ltil_X$ coincides with the operator $\ltil:\cA\rightarrow \chU$ defined in \cite[10.1.3]{b-KS}, see also \cite[1.3]{a-KolbStok09}. The operator $\ltil$ has a second variant $\ltil'=\ltil_Y$ for $Y=\Rmat_{21}^{-1} \Rmat^{-1}$, which will be more relevant in our setting, see Equation \eqref{eq:inadUl} in Proposition \ref{prop:Z-construct}.
\end{rema}
Property \eqref{eq:hopf-contraction} of the contraction pairing implies that
\begin{align}\label{eq:ltilXY}
  \ltil_{XY}(a)=\ltil_X(a_{(1)})\,\ltil_Y(a_{(2)}) \qquad \mbox{for all $X,Y\in \sU^{(2)}_{0,\sigma}$, $a\in \cA_\sigma$.}
\end{align}
For later reference we record that the relations $\Rmat^{-1}=(S\ot \id)(\Rmat)$ and $(S\ot S)(\Rmat)=\Rmat$ imply that
\begin{align}\label{eq:ltil-rels}
    \ltil_{\Rmat_{21}}(S(a))=\ltil_{\Rmat_{21}^{-1}}(a), \qquad \ltil_\Rmat(S^{-1}(a))=\ltil_{\Rmat^{-1}}(a)
\end{align}
for any $a\in \cA_\sigma$.
\subsection{The 2-tensor universal $K$-matrix for $\Bcs$}
Inspired by Lemma \ref{lem:KcalK} we now define an element $\scrK\in\sU^{(2)}_{0,\sigma}$ by
\begin{align}\label{eq:KBU-def}
  \scrK=\Rmat_{21}\cdot(1\ot \cK)\cdot\Rmat
\end{align}
where $\cK$ is the 1-tensor universal $K$-matrix defined by \eqref{eq:K-def}.
Formulas \eqref{eq:kowK-qsp} and \eqref{eq:kowK-qsp2} for the coproduct of $\cK$ imply that
\begin{align}\label{eq:KinK}
  \scrK=\kow(\cK)\cdot(\cK^{-1}\ot 1)=(\cK^{-1}\ot 1)\cdot\kow(\cK).
\end{align}
The following proposition provides the main ingredient needed to show that $\Kmat \in \sB^{(2)}$. Define $R_X^{TT}=(T_{w_X}^{-1}\ot T_{w_X}^{-1})(R_X)$ to simplify notation.
\begin{prop}\label{prop:KinBU}
  The element $\scrK$ defined by {\upshape\eqref{eq:KBU-def}} satisfies the relation
  \begin{align}
     \scrK &=R^\theta\cdot R_X^{TT}\cdot \kappa^-\cdot\kappa^{-\Theta}\cdot \big( 1 \ot \xi  T_{w_X}^{-1} T_{w_0}^{-1} \sigma \big). \label{eq:tautau0idK}
  \end{align} 
\end{prop}
\begin{proof}
  Recall that
  \begin{align}
    \kow(T_{w_0}) &= (T_{w_0}\ot T_{w_0}) \cdot R^{-1}\label{eq:kowTw0}\\
    \kow(T_{w_X}^{-1})&= (T_{w_X}^{-1}\ot T_{w_X}^{-1})\cdot \kappa \cdot R_{X\, 21} \cdot \kappa^-, \label{eq:kowTwX-}
  \end{align}
  see for example \cite[Lemma 3.8]{a-BalaKolb15p}. By \eqref{eq:RuuopR} we have
  \begin{align}
    \kow(T_{w_X}^{-1})\cdot R &\stackrel{\phantom{\eqref{eq:kowTwX-}}}{=}  R \cdot\kappa^- \cdot \kow^\op(T_{w_X}^{-1}) \cdot \kappa \nonumber\\
        &\stackrel{\eqref{eq:kowTwX-}}{=} R \cdot \kappa^- \cdot (T_{w_X}^{-1}\ot T_{w_X}^{-1})\cdot \kappa \cdot R_X\nonumber\\
        &\stackrel{\phantom{\eqref{eq:kowTwX-}}}{=} R\cdot R_X^{TT}\cdot (T_{w_X}^{-1}\ot T_{w_X}^{-1}).\label{eq:kowTwX-R}
  \end{align}
   With these preparations we calculate
  \begin{align*}
    \kow(\cK)\cdot(\cK^{-1}\ot 1)& \stackrel{\phantom{\eqref{eq:kowTwX-}}}{=}\kow(\Xfrak \xi T_{w_X}^{-1} T_{w_0}^{-1}\sigma) \cdot \big(\sigma T_{w_0} T_{w_X} \xi^{-1} \Xfrak^{-1} \ot 1 \big)\\
    &\stackrel{\eqref{eq:kowTw0}}{=}\kow(\Xfrak) \cdot \kow(\xi) \cdot \kow(T_{w_X}^{-1}) \cdot R  
       \cdot\big(T_{w_X} \xi^{-1} \Xfrak^{-1} \ot T_{w_0}^{-1} \sigma\big)\\
    &\stackrel{\eqref{eq:kowTwX-R}}{=}\kow(\Xfrak)\cdot \kow(\xi)\cdot R  \cdot R_X^{TT} \cdot \big(  \xi^{-1} \Xfrak^{-1} \ot T_{w_X}^{-1} T_{w_0}^{-1} \sigma\big)\\
    &\stackrel{\eqref{eq:kow-xi}}{=} \kow(\Xfrak)\cdot R  \cdot R_X^{TT}\cdot (\xi\ot \xi) \cdot \kappa^-\cdot\kappa^{-\Theta} \cdot \big(  \xi^{-1} \Xfrak^{-1} \ot T_{w_X}^{-1} T_{w_0}^{-1} \sigma\big)\\
    &\stackrel{\phantom{\eqref{eq:kow-xi}}}{=} \kow(\Xfrak)\cdot R  \cdot R_X^{TT}\cdot \kappa^- \cdot 
    \kappa^{-\Theta} \cdot \big( \Xfrak^{-1} \ot \xi T_{w_X}^{-1} T_{w_0}^{-1}\sigma \big).
  \end{align*}  
  By \cite[Proposition 6.1.(4)]{a-BalaKolb15p} the element $\Xfrak\ot 1$ commutes with $\kappa^-\cdot\kappa^{-\Theta}$, and by \eqref{eq:Xintertwine} it also commutes with $R_X^{TT}$. Hence one gets
  \begin{align*}   
     \kow(\cK)\cdot(\cK^{-1}\ot 1)&=\kow(\Xfrak)\cdot R \cdot(\Xfrak^{-1}\ot 1) \cdot R_X^{TT}\cdot \kappa^- \cdot 
    \kappa^{-\Theta} \cdot \big(1\ot \xi T_{w_X}^{-1} T_{w_0}^{-1} \sigma\big)
  \end{align*} 
which proves formula \eqref{eq:tautau0idK}.
\end{proof}
We are now in a position to show that the element $\scrK$ defined by \eqref{eq:KBU-def} is a universal K-matrix for $\Bcs$ in the sense of Definition \ref{def:K-matrix} up to completion. 
\begin{thm}\label{thm:scrK}
  The element $\scrK\in \sU^{(2)}_{0,\sigma}$ defined by \eqref{eq:KBU-def} has the following properties:
  \begin{enumerate}
    \setcounter{enumi}{-1}
    \item $\scrK\in \sB^{(2)}$,\label{item0-srcK}
    \item $\scrK \cdot\kow(b)  = \kow(b) \cdot \scrK $ for all $b\in \Bcs$,
    \item $(\kow\ot \id)(\scrK)=\Rmat_{32}\cdot \scrK_{13} \cdot \Rmat_{23}$,
    \item $(\id \ot \kow)(\scrK)= \Rmat_{32}\cdot\scrK_{13} \cdot\Rmat_{23} \cdot\scrK_{12}$.
  \end{enumerate}
\end{thm}
\begin{proof}
To prove property (0), recall from Examples \ref{eg:kkT} and \ref{eg:RthetainB2} that $\kappa^-\cdot \kappa^{-\Theta}\in \sB^{(2)}$ and that $R^\theta\in \sB^{(2)}$. Moreover, $\langle R_X^{TT},a\rangle_{2,1}\in \uqgX \subset \Bcs$ for all $a\in \cA_\sigma$ and hence $R_X^{TT}\in \sB^{(2)}$. As $\sB^{(2)}$ is a subalgebra of $\sU^{(2)}_{0,\sigma}$ we obtain $\Kmat\in \sB^{(2)}$ from Proposition \ref{prop:KinBU}.

Properties (1), (2) and (3) of the theorem follow from the definition \eqref{eq:KBU-def} and from Theorem \ref{thm:BK} in the same way as the second part of Lemma \ref{lem:KcalK}. 
\end{proof}
\subsection{The braided module category of finite dimensional $\Bcs$-modules}\label{sec:braid-mod-cat}
Let $\Mcs$ denote the category of finite dimensional $\Bcs$-modules. As $\Bcs$ is a right coideal subalgebra of $\Uc$, the category $\Mcs$ is a right module category over $\Oint$ and hence over $\Oint^{\langle\sigma\rangle}$. Recall from Proposition \ref{cor:Os-rib} that $(\Oint^{\langle \sigma\rangle},\ot,{\mathbbm 1},c)$ is a braided monoidal category with braiding $c$ given by \eqref{eq:braiding}. We want to show that multiplication by the element $\Kmat$ endows $\Mcs$ with the structure of a braided module category over $\Oint^{\langle \sigma \rangle}$. To this end it remains to show that multiplication by an element in $\sB^{(2)}$ provides a natural transformation of $\ot:\Mcs\times \Oint\rightarrow \Mcs$. The subtlety lies here in the fact that elements of $\sB^{(2)}\subset \sU^{(2)}_{0,\sigma}$ are a priori only defined on tensor products $V_1\ot V_2$ for $V_1, V_2\in Ob(\Oint^{\langle \sigma \rangle})$. However, we would like to let $\Kmat$ act on tensor products $M\ot V$ where $M\in Ob(\Mcs)$.

We address this problem in a slightly more general setting. For any $n\in \N$ consider the forgetful functor
\begin{align*}
  \For_B^{(n)}:\Mcs \times \underbrace{\Oint^{\langle \sigma \rangle}\times \dots \times \Oint^{\langle \sigma \rangle}}_{\text{$n{-}1$ times}} \longrightarrow \mathcal{V}ect\\
  (M,V_1,\dots,V_{n-1})\mapsto M\ot V_1\ot\dots\ot V_{n-1}
\end{align*}
into the category of $\field(q^{1/d})$-vector spaces, and define
\begin{align}
  \scrC^{(n)}=\End(\For_B^{(n)}).
\end{align}
Elements of $\scrC^{(n)}$ are defined on $M\ot V_1\ot\dots\ot V_{n-1}$ where $M\in Ob(\Mcs)$ and $V_i\in Ob(\Oint^{\langle \sigma \rangle})$ for $i=1,\dots,n-1$.
Let $k,n\in \N$ with $1\le k\le n$. Similar to \eqref{eq:kowknU} there is an algebra homomorphism
\begin{align}\label{eq:kowknC}
  \kow^{\scrC,\sigma}_{n,k}: \scrC^{(n)} \rightarrow \scrC^{(n+1)}
\end{align}
defined by
\begin{align*}
  \big(\kow^{\scrC,\sigma}_{n,k}(\varphi)\big)_{M,V_1,\dots,V_n}=
    \begin{cases}
      \varphi_{\raisebox{-1pt}{\tiny $M\ot V_1,V_2,\dots,V_n$}}& \mbox{if $k=1$,}\\
      \varphi_{M,V_1, \dots, V_{k-2}, V_{k-1}\ot V_k, V_{k+1}, \dots, V_n} &\mbox{if $k\ge 2$.}
    \end{cases}
\end{align*}
For any $n\in\N$ and any $V_1, \dots, V_{n-1}\in Ob(\Oint^{\langle \sigma \rangle})$ we use Sweedler notation for the $\cA^{\ot(n-1)}_\sigma$-comodule structure on $V_1\ot \dots \ot V_{n-1}$. More explicitly, for $v\in V_1\ot \dots \ot V_{n-1}$ we write
\begin{align*}
  \kow(v)=v_{(0)}\ot v_{(1)} \in \big(V_1\ot \dots \ot V_{n-1}\big) \ot \cA^{\ot(n-1)}_\sigma. 
\end{align*}
With this notation we define a linear map
\begin{align*}
  \phi_n:\sB^{(n)} \rightarrow \scrC^{(n)}
\end{align*}
by the property that for any $X\in \sB^{(n)}$ one has
\begin{align}\label{eq:phin-def}
  \phi_n(X)(m\ot v)=\langle X,v_{(1)}\rangle_{n,n-1}(m)\ot v_{(0)}
\end{align}
for all $m\in M$, $v\in V_1\ot \dots\ot V_{n-1}$ where $M\in Ob(\Mcs)$ and $V_1,\dots,V_{n-1}\in Ob(\Oint^{\langle \sigma \rangle})$. The following proposition for $n=2$ allows us to consider $\Kmat$ as an element in $\scrC^{(2)}$ and hence as a natural transformation of $\ot:\Mcs\times \Oint^{\langle \sigma \rangle}\rightarrow \Mcs$.
\begin{prop}\label{prop:phin}
  For any $n\in \N$ the map $\phi_n:\sB^{(n)} \rightarrow \scrC^{(n)}$ is an injective algebra homomorphism. Moreover,
  \begin{align}\label{eq:kowkn-compatible}
    \phi_{n+1}\circ \kow^{\sU,\sigma}_{n,k} = \kow^{\scrC,\sigma}_{n,k}\circ \phi_n
  \end{align}
  for all $1\le k\le n$.
\end{prop}
\begin{proof}
  Let $X,Y\in \sB^{(n)}$ and $m$, $v$ as in \eqref{eq:phin-def}. Equation \eqref{eq:hopf-contraction} implies that 
  \begin{align*}
    \phi_n(XY)(m\ot v)&=\big(\langle XY,v_{(1)}\rangle_{n,n-1}(m)\big)\ot v_{(0)}\\
      &= \big(\langle X,v_{(1)}\rangle_{n,n-1}  \langle Y,v_{(2)}\rangle_{n,n-1} (m)\big)\ot v_{(0)}\\
      &=\phi_n(X)\big(\phi_n(Y)(m\ot v) \big).
  \end{align*}
This shows that $\phi_n$ is an algebra homomorphism. Any element $X\in \sB^{(n)}$ is uniquely determined by its action on all $V_1\ot \dots \ot V_n$ where $V_i\in Ob(\Oint^{\langle \sigma \rangle})$. As $V_1$ can be considered as a $\Bcs$-module by restriction, the map $\phi_n$ is injective. Finally, formula \eqref{eq:kowkn-compatible} follows by comparison of the definitions of the maps \eqref{eq:kowknU} and \eqref{eq:kowknC}. 
\end{proof}
The above proposition implies that $\Mcs$ is a braided module category over $\Oint^{\langle \sigma \rangle}$. Indeed, by definition of $\scrC^{(2)}$ the element 
\begin{align*}
  \phi_2(\Kmat)=\{\phi_2(\Kmat)_{M,V}:M\ot V\rightarrow M\ot V\}_{M\in Ob(\Mcs),V\in Ob(\Oint^{\langle \sigma\rangle})}
\end{align*}
is natural in both tensor entries. Moreover, property (1) in Theorem \ref{thm:scrK} implies that $\phi_2(\Kmat)_{M,V}$ is a $\Bcs$-module homomorphism for any $M\in Ob(\Mcs)$ and $V\in Ob(\Oint^{\langle \sigma \rangle})$.
\begin{cor}\label{cor:B-mod-braided}
  The triple $(\Mcs,\ot,\phi_2(\Kmat))$ is a braided module category over $(\Oint^{\langle \sigma \rangle},\ot,\mathbbm{1}, c)$.
\end{cor}
\begin{proof}
  As explained above $\phi_2(\Kmat)$ defines an element in $\Aut(\ot)$ for the tensor product $\ot:\Mcs\times \Oint^{\langle \sigma \rangle}\rightarrow \Mcs$ which comes from the coideal structure of $\Bcs$.
  Properties \eqref{eq:braidedMod1} and \eqref{eq:braidedMod2} of Definition \ref{def:braidedMod} for $e=\phi_2(\Kmat)$ follow from formulas (2) and (3) in Theorem \ref{thm:scrK} together with formula \eqref{eq:kowkn-compatible} in Proposition \ref{prop:phin}.
\end{proof}
\begin{rema}
If $\tau=\tau_0$ then $\Oint^{\langle \sigma \rangle}=\Oint$. Hence in this case $(\Mcs,\ot,\phi_2(\Kmat))$ is a braided module category over $(\Oint,\ot,\mathbbm{1}, c)$.
\end{rema}
\section{The center of $\Bcs$ revisited}\label{sec:center}
As an application we now use the universal K-matrix for quantum symmetric pairs to obtain a distinguished basis of the center $Z(\Bcs)$ of $\Bcs$. The construction consists of two steps. In the first step we show that $Z(\Bcs)$ inherits a direct sum decomposition from the locally finite part of $\Uc$, and we estimate the dimensions of the direct summands. The constructions leading up to Proposition \ref{prop:dim-estimate} are already contained in \cite{a-KL08}, but we present them here in a much streamlined fashion. The second step, which was the hard part in \cite{a-KL08}, is to prove the existence of nontrivial central elements in sufficiently many components of the locally finite part. The construction of such elements is now much simplified by the existence of the $2$-tensor universal K-matrix $\Kmat$ and the graphical calculus for braided module categories. Eventually, the generalized $l$-operator for the 2-tensor universal K-matrix provides a surjective algebra homomorphism from the invariants in the restricted dual of $\Uq_\sigma$ to the center $Z(\Bcs)$. This allows us in particular to establish the multiplication formula \eqref{eq:LR} which was beyond reach in \cite{a-KL08} and \cite{a-KolbStok09}
\subsection{The locally finite part $F_l(\Uc)$}
The Hopf algebra $\Uc$ acts on itself by the left adjoint action
\begin{align*}
  \ad(u)(x) = u_{(1)} x S(u_{(2)}) \qquad \mbox{for all $x,u\in \Uc$.}
\end{align*}
We write
\begin{align*}
  F_l(\Uc)=\{x\in \Uc \,|\, \dim(\ad(\Uc)(x))<\infty\}
\end{align*}
to denote the locally finite part of $\Uc$ with respect to this action. By \cite[Theorem 4.10]{a-JoLet2} the left locally finite part is a direct sum of cyclic modules
\begin{align}\label{eq:Fl-PW}
  F_l(\Uc)=\bigoplus_{\lambda\in P^+} \ad_l(\Uc)(K_{-2\lambda}).
\end{align}
Recall that $\cA\subset \check{U}^\ast$ denotes the Hopf algebra generated by the matrix coefficients of all representations in $\Oint$.
The direct sum decomposition \eqref{eq:Fl-PW} is reminiscent of the Peter-Weyl decomposition of $\cA$. To make this precise, let $\ad_r(u)(x)=S(u_{(1)})x u_{(2)}$ denote the right adjoint action of $\Uc$ on itself. The Hopf algebra $\cA$ is a left $\Uc$-module with the dual action $\ad_r^\ast$ given by
\begin{align*}
  \big(\ad_r^\ast(u)(c)\big)(x) = c\big(\ad_r(u)(x)\big) \qquad \mbox{for all $c\in \cA$ and $u, x\in\Uc$.}
\end{align*}
Recall that $V(\lambda)$ denotes the simple $\Uc$-module of highest weight $\lambda \in P^+$. For any finite-dimensional $\Uc$-module $V$ the dual space $V^\ast$ has a left $\Uc$-module structure given by
\begin{align*}
  (uf)(v)=f(S(u)v) \qquad \mbox{for all $f\in V^\ast$, $v\in V$, $u\in \Uc$.}
\end{align*}
With respect to this $\Uc$-module structure one has
\begin{align}\label{eq:Vldual}
  V(\lambda)^\ast \cong V(-w_0\lambda) \qquad \mbox{for all $\lambda\in P^+$.}
\end{align}
Let $C^\lambda = \mathrm{span} \{c_{f,v}\,|\,f\in V(\lambda)^\ast, v\in V(\lambda)\}$ denote the $\field(q^{1/d})$-linear span of the matrix coefficients corresponding to the simple $\Uc$-module $V(\lambda)$. We also write $c^\lambda_{f,v}$ instead of $c_{f,v}$ if we want to stress that $c_{f,v}\in C^\lambda$. 

To describe the left $\Uc$-module structure of the direct summand $\ad(\Uc)(K_{-2\lambda})$ we use the $l$-operators $\ltil, \ltil':\cA\rightarrow \Uc$ from Remark \ref{rem:ltil}. Recall that $\ltil(a)=\langle\Rmat\Rmat_{21},a\rangle_{2,1}$ and $\ltil'(a)=\langle\Rmat_{21}^{-1} \Rmat^{-1},a\rangle_{2,1}$ for all $a\in \cA$. The following proposition is in principle contained in \cite{a-Caldero93}, see also \cite[Remark 1.6, Proposition 1.7]{a-KolbStok09} for the exact statement given here.
\begin{prop}\label{prop:Caldero}
  The $l$-operators $\ltil$ and $\ltil'$ define isomorphisms of left $\Uc$-modules $\ltil, \ltil':\cA\rightarrow F_l(\Uc)$ with   
\begin{align*}
   \ltil(C^\lambda)=\ad(\Uc)(K_{-2\lambda})\quad \mbox{and} \quad  \ltil'(C^\lambda)=\ad(\Uc)(K_{2w_0\lambda}) \qquad \mbox{for all $\lambda\in P^+$.}
\end{align*}
Moreover, if $v_\lambda\in V(\lambda)_\lambda$ is a highest weight vector and $f_{-\lambda}\in V(\lambda)^\ast_{-\lambda}$ is a lowest weight vector with $f_{-\lambda}(v_\lambda)=1$, then
\begin{align*}
  \ltil(c^\lambda_{f_{-\lambda},v_\lambda}) = K_{-2\lambda}.
\end{align*}
\end{prop}
With the identification \eqref{eq:Vldual} the map
\begin{align*}
  V(-w_0\lambda)\ot V(\lambda) \rightarrow C^\lambda, \qquad f\ot v\mapsto c_{f,v}^\lambda
\end{align*} 
is an isomorphism of left $\Uc$-modules for any $\lambda\in P^+$. Hence Proposition \ref{prop:Caldero} implies that there is an isomorphism of left $\Uc$-modules 
\begin{align}\label{eq:phi-iso}
  \varphi:\ad_l(\Uc)(K_{-2\lambda}) \rightarrow V(-w_0\lambda)\ot V(\lambda)
\end{align}
with $\varphi(K_{-2\lambda})=f_{-\lambda}\ot v_\lambda$.

In Section \ref{sec:cent-inside} we identify lowest weight components of elements in the center of $\Bcs$ which also lie in a given Peter-Weyl block $\ad(\Uc)(K_{-2\lambda})$. The following preparatory Lemma describes the image under $\varphi$ of these lowest weight components.

\begin{lem}\label{lem:phi}
  Let $Y\subseteq I$ be any subset and $\lambda\in P^+$.
  \begin{enumerate}
    \item Assume that $(w_0\lambda,\alpha_i)=(w_Y\lambda,\alpha_i)$ for all $i\in Y$. In this case there exists an up to scalar multiplication uniquely determined lowest weight vector $v_\mu\in V(-w_0\lambda)\ot V(\lambda)$ of weight $\mu\in -P^+$ with the following properties
      \begin{enumerate}
        \item  $v_\mu \in U_q(\gfrak_Y)f_{-\lambda}\ot V(\lambda)$, 
        \item  $u v_\mu=\vep(u) v_\mu$ for all $u\in U_q(\gfrak_Y)$. 
      \end{enumerate}
      Moreover, in this case $\mu=w_0\lambda-w_Y\lambda$.
    \item Assume that $(w_0\lambda,\alpha_i)\neq(w_Y\lambda,\alpha_i)$ for some $i\in Y$. In this case there does not exist a lowest weight vector $v_\mu\in V(-w_0\lambda)\ot V(\lambda)$ which satisfies properties {\upshape (a)} and {\upshape (b)} above.
  \end{enumerate}
\end{lem}
\begin{proof}
  For any lowest weight vector $v_\mu\in V(-w_0\lambda) \ot V(\lambda)$ of weight $\mu\in -P^+$ there exists a nonzero element $a\in V(-w_0\lambda)$ such that
  \begin{align*}
    v_\mu - a \ot v_{w_0\lambda} \in V(-w_0\lambda)\ot \bigoplus_{\nu>w_0\lambda} V(\lambda)_\nu.
  \end{align*}
  If additionally (a) holds, then this implies that
  \begin{align*}
    \mu- w_0\lambda + \lambda \in Q^+_Y= \sum_{i\in Y}\N_0 \alpha_i.
  \end{align*}
  Hence $v_\mu\in U_q(\gfrak_Y) f_{-\lambda} \ot U_q(\gfrak_Y) v_{w_0\lambda}$, in other words, $v_\mu$ lies in the tensor product of two simple $U_q(\gfrak_Y)$-modules with lowest weight vectors $f_{-\lambda}$ and $v_{w_0\lambda}$, respectively. Property (b) states that $v_\mu$ spans a trivial $U_q(\gfrak_Y)$-submodule in this tensor product. Such a trivial submodule exists if and only if $(w_0\lambda,\alpha_i)=(w_Y\lambda,\alpha_i)$ for all $i\in Y$. Moreover, in this case $v_\mu$ has weight $w_0\lambda-w_Y\lambda$.
\end{proof}
\subsection{Central elements of $\Bcs$ inside $\ad(\Uc)(K_{-2\lambda})$}\label{sec:cent-inside}
Let $Z(\Bcs)$ denote the center of the algebra $\Bcs$. As $\Bcs$ is a right coideal subalgebra of $\Uc$ one has
\begin{align}\label{eq:centerB}
  Z(\Bcs)= \{x\in \Bcs\,|\, \ad(b)(x) = \vep(b)x \mbox{ for all $b\in \Bcs$}\},
\end{align}
see for example \cite[Lemma 4.11]{a-KolbStok09}.
G.~Letzter showed in \cite[Theorem 1.2]{a-Letzter-memoirs} that hence
\begin{align}\label{eq:inFl}
  Z(\Bcs) \subset F_l(\Uc),
\end{align}
see also \cite[Proposition 4.12]{a-KolbStok09}. As $K_{-2\lambda}$ is grouplike, the Peter Weyl summand $\ad(\Uc)(K_{-2\lambda})$ is a left coideal of $\Uc$. Hence the inclusion \eqref{eq:inFl} implies that 
\begin{align}\label{eq:center-decomposition}
  Z(\Bcs) = \bigoplus_{\lambda\in P^+} Z(\Bcs) \cap \ad(\Uc)(K_{-2\lambda}).
\end{align}
This means that we only need to understand central elements of $\Bcs$ which are contained in one single Peter-Weyl summand $\ad(\Uc)(K_{-2\lambda})$.

We recall some notation from \cite{a-Kolb14} which will be convenient in the proof of the next lemma. For any multi-index $J=(j_1,j_2,\dots,j_n)\in I^n$ we define $F_J=F_{j_1} F_{j_2} \dots F_{j_n}$ and $B_J=B_{j_1} B_{j_2} \dots B_{j_n}$ and $\alpha_J=\sum_{k=1}^n \alpha_{j_k}$. Let $\cJ$ be a fixed subset of $\cup_{n\in \N_0}I^n$ such that $\{F_J\,|\,J \in \cJ\}$ is a basis of $U^-$. By \cite[Proposition 6.2]{a-Kolb14} any element $z\in \Bcs$ can be written uniquely in the form
\begin{align}\label{eq:zmB}
  z=\sum_{J\in \cJ} m_J B_J
\end{align} 
where $m_J\in \cM_X^+ \Uc^0_\Theta$ and all but finitely many $m_J$ vanish. 
Let $Q_X=\bigoplus_{i\in X}\Z \alpha_i$ be the subgroup of the root lattice $Q$ generated by all $\alpha_i$ for $i\in X$.
For any $\alpha=\sum_{i\in I} n_i\alpha_i\in Q$ let $\overline{\alpha}^X\in Q/Q_X$ denote the coset defined by $\alpha$. 
\begin{lem}\label{lem:zmu}
  Let $\lambda\in P^+$. Any nonzero element $z\in Z(\Bcs)\cap \ad(\Uc)(K_{-2\lambda})$ can be written uniquely as a sum
  \begin{align}\label{eq:z=zm+lot}
    z=z_\mu + \mathrm{lot}
  \end{align}
  with the following properties:
  \begin{enumerate}
    \item $0\neq z_\mu \in \big(\ad(\Uc)(K_{-2\lambda})\big)_{\mu}$ and $\lot \in \bigoplus_{\nu\neq\mu} \big(\ad(\Uc)(K_{-2\lambda})\big)_{\nu}$,
    \item $\displaystyle z_\mu \in \bigoplus_{\nu\in Q^+, \overline{\nu}^X=\overline{\lambda-w_0\lambda}^X}\cM_X^+\UnullTh U^-_{-\nu}$,
    \item $\ad(u)(z_\mu) = \vep(u)z_\mu$ for all $u\in \cM_X^+ \Uc^0_\Theta U^- $.
  \end{enumerate}
  Moreover, in the above situation $\mu=w_0\lambda - w_X \lambda$, and the one-dimensional space $\field(q^{1/d})z_\mu$ spanned by $z_\mu$ is independent of the element $z$.
\end{lem}
\begin{proof}
  Given a nonzero element $z\in  Z(\Bcs)\cap \ad(\Uc)(K_{-2\lambda})$ write $z=\sum_{J\in \cJ}m_J B_J$ as in \eqref{eq:zmB} with $m_J\in \cM_X^+\UnullTh$. Choose $J_{max}\in \cJ$ such that $\alpha_{J_{max}}$ is maximal among all $\alpha_J$ with $m_J\neq 0$. Now define
  \begin{align*}
    z_\mu = \sum_{\{J\in \cJ\,|\, \overline{\alpha}_J=\overline{\alpha}_{J_{max}}\}} m_J F_J.
  \end{align*}
Let $U^0_X$ denote the subalgebra generated by all $K_i, K_i^{-1}$ for $i\in X$  
As $z$ is $\ad(\Bcs)$-invariant, each of the $m_J$ is a weight vector for $U^0_X$ of weight $-\alpha_J$. This implies that $z_\mu$ is a weight vector for the adjoint action of $\Uc$. Moreover, the $\ad(\Bcs)$-invariance of $z$ implies that $z_\mu$ is an $\ad(\cM_X^+ \UnullTh)$-invariant lowest weight vector for the adjoint action. This proves property (3).

Define $\lot=z-z_\mu$. Then \eqref{eq:z=zm+lot} holds and property (1) holds by construction. Moreover, $z_\mu \in
\cM_X^+ \UnullTh U^- \cap \ad_l(\Uc)(K_{-2\lambda})$. Hence, applying the isomorphism $\varphi$ from \eqref{eq:phi-iso} we obtain
\begin{align*}
  \varphi(z_\mu) \in \uqgX f_{-\lambda} \ot V(\lambda).
\end{align*}
Hence Lemma \ref{lem:phi} for $Y=X$ implies that the weight $\mu$ of $z_\mu$ is given by $\mu=w_0\lambda-w_X\lambda$. This proves (2). Finally, Lemma \ref{lem:phi} also implies that the nonzero element $\varphi(z_\mu)$ is uniquely determined up to an overall factor. This completes the proof of the Lemma.
\end{proof}
Recall from Section \ref{sec:qsp} that we write $\sigma=\tau\tau_0$.
\begin{prop}\label{prop:dim-estimate}
  For any $\lambda\in P^+$ one has
  \begin{align*}
    \dim\big(Z(\Bcs) \cap \ad(\Uc)(K_{-2\lambda}) \big) &= 0 & \mbox{if $\sigma(\lambda)\neq \lambda$,}\\
    \dim\big(Z(\Bcs) \cap \ad(\Uc)(K_{-2\lambda}) \big) &\le 1 & \mbox{if $\sigma(\lambda)= \lambda$.}
  \end{align*}
\end{prop}
\begin{proof}
  The fact that the nonzero element $z_\mu$ in Lemma \ref{lem:zmu} is independent of $z$ up to an overall factor implies that $\dim\big(Z(\Bcs) \cap \ad(\Uc)(K_{-2\lambda}) \big) \le 1$ for all $\lambda\in P^+$.
  
  Assume now that there exists a nonzero element $z\in Z(\Bcs) \cap \ad(\Uc)(K_{-2\lambda})$. Write $z=z_\mu+\lot$
  where $\mu=-w_0\lambda+w_X\lambda$ and $z_\mu$ satisfies properties (1), (2), and (3) of Lemma \ref{lem:zmu}.
  In  view of  the triangular decomposition of $\Uc$, Properties (1) and (2) imply that
  \begin{align}\label{eq:obstruction}
    z_\mu \in \Big( \bigoplus_{\nu\in Q^+, \overline{\nu}^X=\overline{\lambda-w_0\lambda}^X} \ad(\cM_X U_{-\nu}^-)(K_{-2\lambda})\Big) \cap \cM_X \UnullTh U^-.
  \end{align}
  Observe that $\ad(U^-_{-\nu})(K_{-2\lambda})\subseteq U^-_{-\nu}K_{-2\lambda+\nu}$. Hence the nonzero element $z_\mu$ can satisfy \eqref{eq:obstruction} only if
  \begin{align*}
    \Theta(-2\lambda + \nu) = -2\lambda + \nu
  \end{align*}
  for some $\nu\in Q^+$ with $\overline{\nu}^X=\overline{\lambda-w_0\lambda}^X$. This condition is equivalent to 
  \begin{align}\label{eq:cond1-tt0}
     \Theta(\lambda+w_0\lambda)=\lambda+w_0\lambda.
  \end{align}
  On the other hand, the fact that $z\in Z(\Bcs)$ is central implies that the highest weight $w_0\mu=\lambda-w_0w_X\lambda$ is spherical, see \cite[Theorem 3.4]{a-Letzter03}, and hence $\Theta(w_0\mu)=-w_0\mu$ or equivalently
  \begin{align}\label{eq:cond2-tt0}
    \Theta(w_0\lambda-w_X\lambda)=-w_0\lambda+w_X\lambda.
  \end{align} 
  By the following lemma conditions \eqref{eq:cond1-tt0} and \eqref{eq:cond2-tt0} imply that $\sigma(\lambda)=\lambda$.
\end{proof}
\begin{lem}\label{lem:PZB}
  Let $\lambda\in P$. The following are equivalent:
  \begin{enumerate}
    \item The weight $\lambda$ satisfies the two relations 
      \begin{align*}
        \Theta(\lambda+w_0\lambda)&=\lambda+w_0\lambda,&
        \Theta(w_0\lambda-w_X\lambda)&=w_X\lambda-w_0\lambda.
      \end{align*}
    \item $\Theta(\lambda) = \lambda+w_0\lambda - w_X\lambda$.
    \item $\sigma(\lambda)=\lambda$
  \end{enumerate}
\end{lem}
\begin{proof}
  {\bf (1) $\Rightarrow$ (2):} Assume that (1) holds. Subtracting the two relations in (1) one obtains
  \begin{align}\label{eq:step1}
    \Theta(\lambda+w_X\lambda) = \lambda-w_X\lambda + 2 w_0\lambda.
  \end{align}
  As $w_X\lambda-\lambda\in Q_X$ one has 
  \begin{align*}
    \Theta(\lambda+w_X\lambda) = \Theta(2\lambda)+\Theta(w_X\lambda-\lambda) = \Theta(2\lambda) + w_X\lambda-\lambda.
  \end{align*}
  Inserting this equation into \eqref{eq:step1} gives (2).
  
 \noindent {\bf (2) $\Leftrightarrow$ (3):} Property (2) is equivalent to
 \begin{align*}
   \Theta(\lambda) - \lambda = -\tau_0(\lambda) + \tau(\lambda) - \tau(\lambda) + \Theta(\tau(\lambda)).
 \end{align*}
 Now the equivalence with property (3) follows from the fact that   
 \begin{align}\label{eq:Tl=Ttl}
   \Theta(\lambda)-\lambda=\Theta(\tau(\lambda))-\tau(\lambda),
 \end{align}  
 see for example \cite[Lemma 3.2]{a-BalaKolb15}.
 
 \noindent {\bf (3) $\Rightarrow$ (1):} If (3) holds then $w_0\lambda= - \tau(\lambda)$ and hence
   \begin{align*} 
     \Theta(\lambda+w_0\lambda) = \Theta(\lambda - \tau(\lambda)) 
                       = \lambda- \tau(\lambda) = \lambda + w_0\lambda
    \end{align*}    
which proves the first equation in (1).    
    Moreover, as shown above, (3) implies (2). Hence we can apply $\Theta$ to the equation in (2) and get
    \begin{align*}
      \Theta(w_0\lambda-w_X\lambda) = \lambda-\Theta(\lambda) \stackrel{\eqref{eq:Tl=Ttl}}{=} \tau(\lambda)- \Theta(\tau(\lambda))=-w_0\lambda+w_X\lambda.
    \end{align*}               
    This proves the second equation in (1).
\end{proof}
\begin{rema}\label{rem:PZB}
  In \cite[Proposition 9.1]{a-KL08} the set 
  \begin{align*}
    P_{Z({\Bcs})}=\{\lambda\in P^+\,|\,\Theta(\lambda) = \lambda + w_0\lambda - w_X\lambda\}
  \end{align*}
  was determined by direct calculation. The equivalence between properties (2) and (3) in Lemma \ref{lem:PZB} gives a conceptual explanation of the formulas obtained in \cite[Proposition 9.1]{a-KL08}. Moreover, the rank of the fixed Lie subalgebra $\kfrak$ coincides with the dimension of the set $\{h\in \hfrak\,|\,\sigma(h)=h\}$, see \cite[1.7, 1.8]{a-Springer87}. This also provides a conceptual proof of \cite[Proposition 9.2]{a-KL08} which states that $P_{Z({\Bcs})}$ is a free additive semigroup with $\rank(P_{Z({\Bcs})})=\rank(\kfrak)$.
\end{rema}
\subsection{The ring of invariants $\cA^\inv_\sigma$}\label{sec:Ainv}
By Proposition \ref{prop:dim-estimate}, to find a basis of the center $Z(\Bcs)$, it remains to construct  a nonzero element in $Z(\Bcs)\cap \ad(\Uc)(K_{-2\lambda})$ for each $\lambda\in P^+$ with $\sigma(\lambda)=\lambda$. We are guided by the graphical calculus \cite{a-RT90}, \cite{b-Turaev} for ribbon categories and its analog for braided module categories \cite{a-Brochier13}. Recall that central elements in $\Uc$ can be found by evaluating the Reshetikhin-Turaev functor on the ribbon tangle in Figure \ref{fig:RR} where the closed strand is colored by a fixed $\Uc$-module $V$ in $\Oint$ and $W$ runs over all modules in $\Oint$.
\begin{figure}[H]
\begin{align*}
  \vcenter{\hbox{\includegraphics[height=2cm]{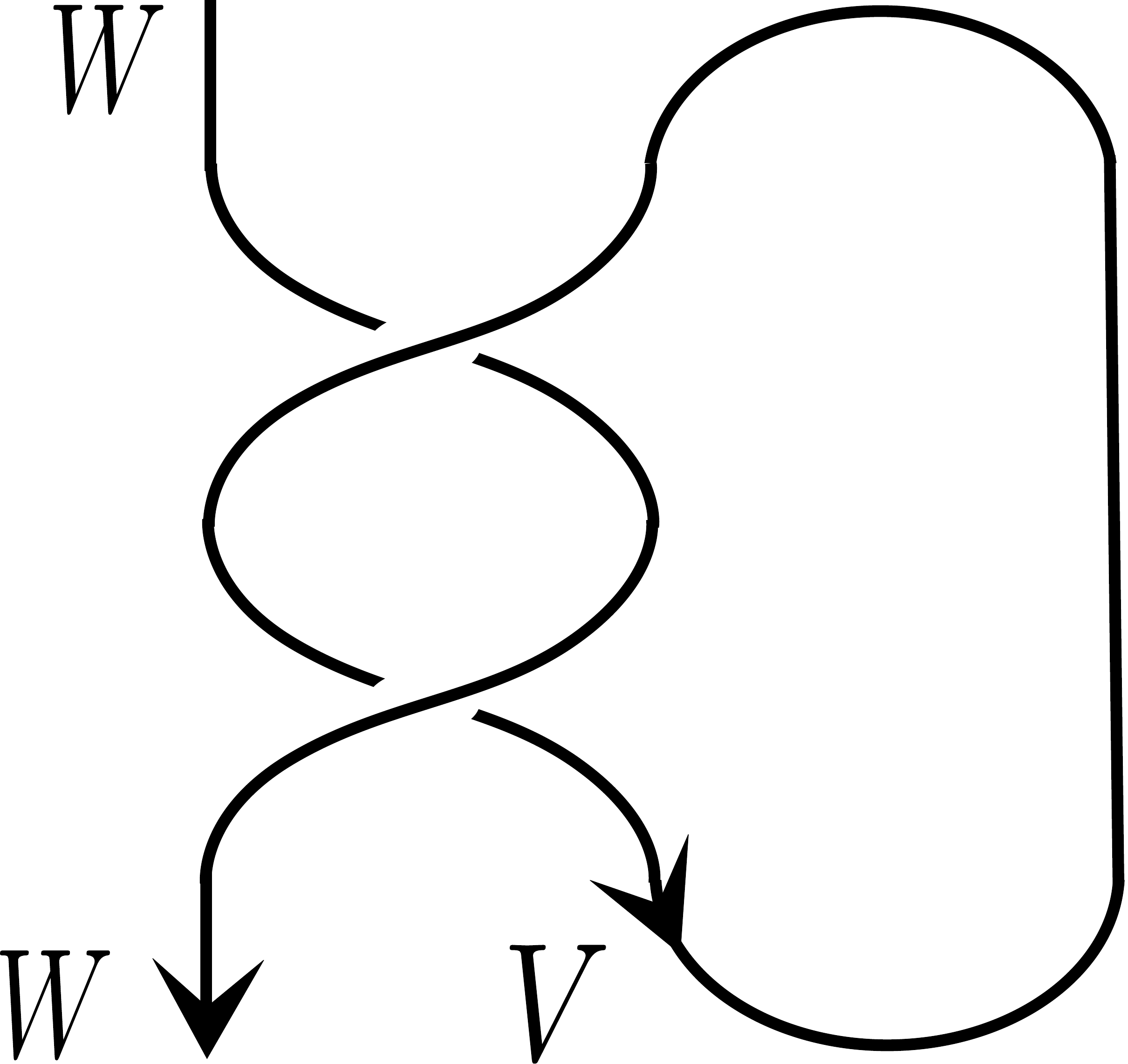}}} 
\end{align*}
\caption{}
\label{fig:RR}
\end{figure}
Algebraically, the corresponding central element is realized by contraction of $\Rmat_{21}\Rmat$ with the quantum trace $\tr_{V,q}\in \cA$ defined by
\begin{align}\label{eq:q-trace}
  \tr_{V,q}(u)=\tr_V(u K_{-2\rho}) \qquad \mbox{for all $u\in \Uc$,}
\end{align} 
see \cite{a-Baumann98} also for further references. In the case of the module category $\Mcs$, consider the ribbon tangle in Figure \ref{fig:K-graph-closed} where the closed strand is colored by a fixed $\Uc_\sigma$-module $V$ in $\Oint^{\langle\sigma \rangle}$ and $M$ runs over all $\Bcs$-modules in $\Mcs$.
\begin{figure}[H]
\begin{align*}
  \vcenter{\hbox{\includegraphics[height=2cm]{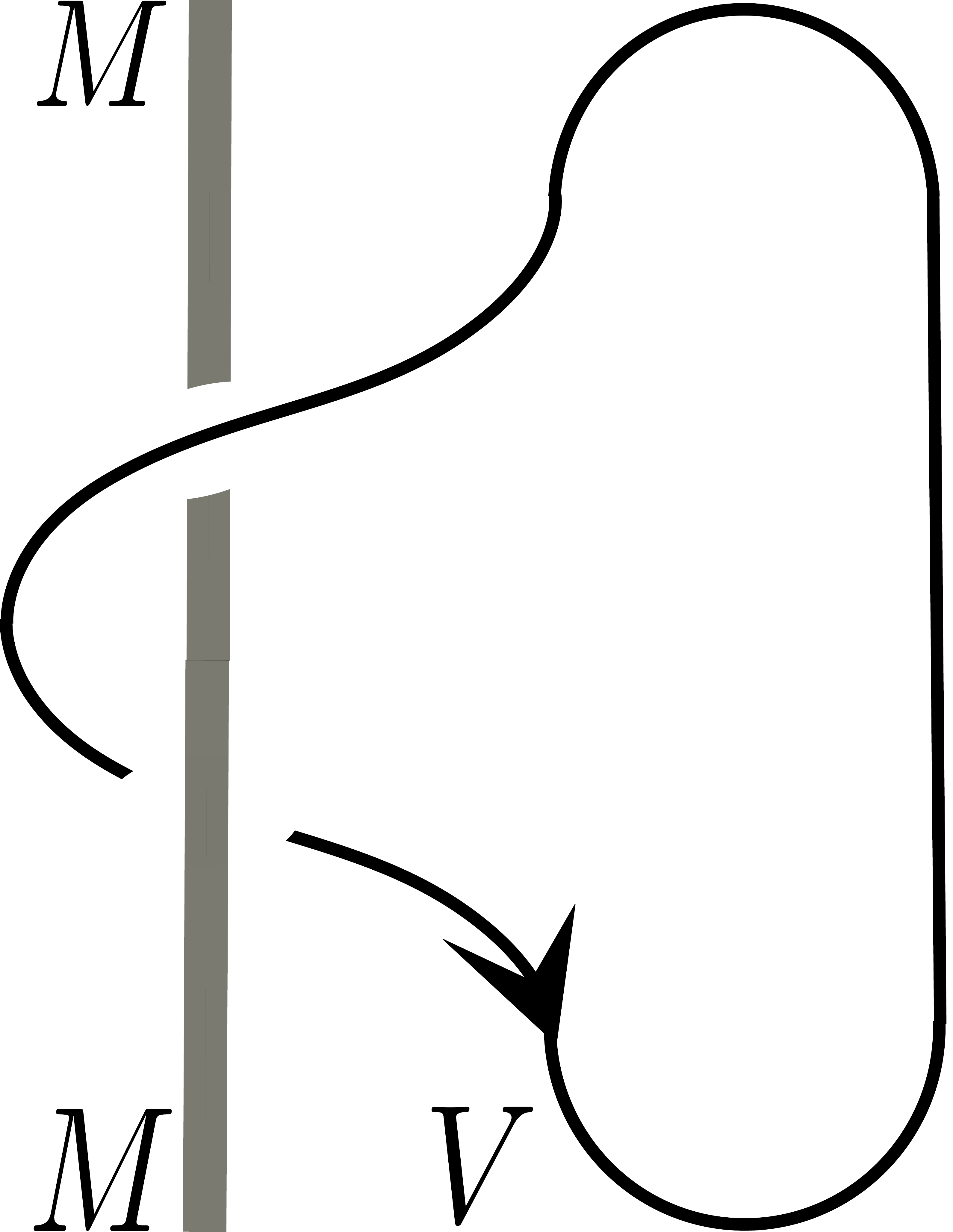}}} 
\end{align*}
\caption{}
\label{fig:K-graph-closed}
\end{figure}
The quantum trace extends to objects $V$ in $\Oint^{\langle \sigma\rangle}$ as the element $\tr_{V,\sigma,q}\in \cA_\sigma$ defined by
\begin{align*}
  \tr_{V,\sigma,q}(u)=\tr_V(u K_{-\rho}) \qquad \mbox{for all $u\in \Uc_\sigma$}.
\end{align*}  
Recall the left adjoint action of $\Uc_\sigma$ on itself. Let $\ad_{l,\sigma}^{\ast}$ denote the induced right adjoint action of $\Uc_\sigma$ on $\cA_\sigma$. More explicitly, one has $\big(\ad_{l,\sigma}^\ast(u)(a)\big)(x)=a\big(u_{(1)}x S(u_{(2)})\big)$ for all $u,x\in \Uc_\sigma$, $a\in \cA_\sigma$.
For any $V$ in $\Oint^{\langle \sigma \rangle}$ the quantum trace $\tr_{V,\sigma,q}\in \cA_\sigma$ is invariant under the right action $\ad_{l,\sigma}^\ast$. Define
\begin{align*}
  \cA_\sigma^\inv = \{a\in \cA_\sigma\,|\,\ad_{l,\sigma}^\ast(u)(a)=\vep(u)a \quad \mbox{for all $u\in \Uc_\sigma$}\}.
\end{align*}
For any $a\in \cA_\sigma^\inv$ and $x,y\in \Uc_\sigma$ we have
$a(x S^2(y))=a(y_{(1)}x S^2(y_{(3)})S(y_{(2)}))=a(yx)$.
Hence
\begin{align}\label{eq:Ainv-property}
  a_{(1)}\ot a_{(2)} = a_{(2)}\ot S^{-2}(a_{(1)}) \qquad \mbox{for all $a\in\cA^\inv_\sigma$.}
\end{align}  
If $\sigma=\id$ then we write  $\cA^\inv$ instead of $\cA^\inv_\id$. In this case the quantum traces $\{\tr_{V(\lambda),q}\,|\,\lambda\in P^+\}$ form a basis of $\cA^\inv$. A similar result holds in the case $\sigma\neq\id$. Recall the simple $\Uc_\sigma$-modules $V(\lambda)_\pm$, $W(\mu)$ defined in Section \ref{sec:equivariantization}.
\begin{lem}\label{lem:Osigma-traces}
  If $\sigma\neq \id$ then the set
  $$\{\tr_{V(\lambda)_\pm,\sigma,q}, \tr_{W(\mu),\sigma,q}\,|\,\lambda, \mu\in P^+, \sigma(\lambda)=\lambda, \sigma(\mu)\neq \mu\}$$
  forms a basis of $\cA_\sigma^\inv$.
\end{lem}  
\begin{proof}
  Recall the set of $\sigma$-orbits $P^+_\sigma$ defined in Section \ref{sec:equivariantization}. By Lemma \ref{lem:Os-ss} the algebra $\cA_\sigma$ has a direct sum decomposition
  \begin{align*}
    \cA_\sigma = \bigoplus_{\lambda\in P^+\atop \sigma(\lambda)=\lambda} (C^{V(\lambda)_+}\oplus C^{V(\lambda)_-}) \bigoplus_{\mu\in P^+_\sigma \atop \sigma(\mu)\neq \mu}C^{W(\mu)}
  \end{align*}
  where $C^V$ denotes the $\field(q^{1/d})$-linear span of the matrix coefficients of the $\Uc_\sigma$-module $V$. In each of the Peter-Weyl blocks the space of $\ad_{l,\sigma}^\ast(\Uc_\sigma)$-invariants is one-dimensional and spanned by the corresponding quantum trace. 
\end{proof}  
\subsection{Construction of central elements in $(\ad_l U)(K_{-2\lambda})$}\label{sec:construction}
 Recall from property (0) in Theorem \ref{thm:scrK} that the element $\scrK$ defined by \eqref{eq:KBU-def} gives rise to a well-defined linear map
\begin{align*}
  \ltil_\scrK:\cA_\sigma \rightarrow \Bcs, \qquad \ltil_\scrK(a)=\langle\scrK,a\rangle_{2,1}.
\end{align*}
By the following proposition and by \eqref{eq:centerB} the image of $\cA^\inv_{\sigma}$ under the map $\ltil_\scrK$ is contained in the center of $\Bcs$.
\begin{prop}\label{prop:inZ}
  Let $\psi\in \cA^\inv_{\sigma}$. Then 
  \begin{align*}
    \ad_l(b)(\ltil_\scrK(\psi))=\vep(b)\ltil_{\scrK}(\psi) \qquad \mbox{for all $b\in \Bcs$}.
  \end{align*}  
  In other words, $\ltil_\scrK(\psi)\in Z(\Bcs)$.
\end{prop}
\begin{proof}
  Let $\psi\in \cA^\inv_{\sigma}$ and $b\in \Bcs$. Using Sweedler notation we calculate
  \begin{align*}
    \ad_l(b)(\ltil_\scrK(\psi))&=\langle(b_{(0)}\ot 1)\cdot\Kmat\cdot(S(b_{(1)})\ot 1), \psi\rangle_{2,1}\\
    &=\langle(b_{(0)}\ot S^{-1}(b_{(2)})b_{(1)})\cdot \Kmat\cdot(S(b_{(3)})\ot 1), \psi\rangle_{2,1}
    \end{align*} 
  Using Theorem \ref{thm:scrK}.(1) one obtains
  \begin{align*}
      \ad_l(b)(\ltil_\scrK(\psi))&=\langle(1\ot S^{-1}(b_{(2)}))\cdot \Kmat\cdot(b_{(0)}S(b_{(3)})\ot b_{(1)}), \psi\rangle_{2,1}\\
      &= \langle (1\ot \ad_{l}(S^{-1}(b_{(1)})) \big(\scrK \cdot (b_{(0)}S(b_{(2)})\ot 1)\big) , \psi \rangle_{2,1}.
  \end{align*}   
  Now we use the assumption $\psi\in \cA^\inv_{\sigma}$ to obtain
  \begin{align*}
    \ad_l(b)(\ltil_\scrK(\psi))&=\vep(b_{(1)})\langle \scrK \cdot (b_{(0)}S(b_{(2)})\ot 1), \psi\rangle_{2,1} 
      =\vep(b) \ltil_\scrK(\psi) 
  \end{align*}
  which completes the proof of the proposition.
\end{proof}
Recall from Remark \ref{rem:Vlp} that in the case $\sigma=\id$ we also write $V(\lambda)_+=V(\lambda)$ for any $\lambda\in P^+$. Let $\lambda\in P^+$ with $\sigma(\lambda)=\lambda$. The next proposition shows that the map $\ltil_\Kmat$ maps the quantum trace for a simple $\Uc_\sigma$-module $V(\lambda)_+$ into the corresponding component of the locally finite part $F_l(\Uc)$. For simplicity we set $a(u)=\langle u, a \rangle_{1,1}$ for all $u\in \sU_\sigma$, $a\in \cA_\sigma$. 
\begin{prop}\label{prop:Z-construct}
  Let $\lambda\in P^+$ with $\sigma(\lambda)=\lambda$. Then the quantum trace $d=\tr_{V(\lambda)_+,\sigma,q}$ satisfies the relation
  \begin{align}\label{eq:inadUl}
    \ltil_\scrK(d) = d_{(1)}(\cK) \ltil_{\Rmat_{21}^{-1} \Rmat^{-1}}(S(d_{(2)})) \in \ad(\Uc)(K_{-2\lambda})\setminus \{0\}.
  \end{align}
\end{prop}
\begin{proof}
  Let $\lambda\in P^+$ with $\sigma(\lambda)=\lambda$ and set $d=\tr_{V(\lambda)_+,\sigma,q}$. We want to determine 
  \begin{align*}
    \ltil_\scrK(d) = \ltil_{\Rmat_{21}(1\ot \cK)\Rmat}(d).
  \end{align*}
 Relation \eqref{eq:ltilXY} implies that
 \begin{align}\label{eq:step2}
    \ltil_\scrK(d) = \ltil_{\Rmat_{21}}(d_{(1)})\,\, d_{(2)}(\cK)\,\, \ltil_{\Rmat}(d_{(3)}).
 \end{align} 
 By \eqref{eq:Ainv-property} the quantum trace $d$ satisfies the relation 
 \begin{align}\label{eq:q-trace-prop}
   d_{(1)} \ot d_{(2)} \ot d_{(3)} = d_{(2)}\ot d_{(3)} \ot S^{-2}(d_{(1)}).
 \end{align}
Inserting the above formula into \eqref{eq:step2} we obtain
\begin{align*}
   \ltil_\scrK(d) \stackrel{\phantom{\eqref{eq:ltil-rels}}}{=}& \ltil_{\Rmat_{21}}(S^2(d_{(3)}))\,\, d_{(1)}(\cK)\,\, \ltil_{\Rmat}(d_{(2)})\\
   \stackrel{\eqref{eq:ltil-rels}}{=}&d_{(1)}(\cK)\,\,
   \ltil_{\Rmat_{21}^{-1}}(S(d_{(3)}))\,\,  \ltil_{\Rmat^{-1}}(S(d_{(2)}))\\
   \stackrel{\eqref{eq:ltilXY}}{=}&d_{(1)}(\cK)\,\,
   \ltil_{\Rmat_{21}^{-1}\Rmat^{-1}}(S(d_{(2)})).
\end{align*} 
 This proves the equality in \eqref{eq:inadUl}. By Proposition \ref{prop:Caldero}, the fact that $d_{(1)}\ot S(d_{(2)})\in C^\lambda\ot C^{-w_0\lambda}$ implies that  $\ltil_\scrK(d)\in \ad(\Uc)(K_{-2\lambda})$.
Moreover, $\ltil_\scrK(d)\neq 0$ because $d_{(1)}(\cK)d_{(2)}\neq 0$ and the map $\ltil'$ in Proposition \ref{prop:Caldero} is an isomorphism. This completes the proof of relation \eqref{eq:inadUl}.
\end{proof}
By Proposition \ref{prop:inZ} the generalized $l$-operator $\ltil_\Kmat$ provides a linear map
\begin{align*}
  \ltil_\Kmat:\cA^\inv_{\sigma}\rightarrow Z(\Bcs)
\end{align*}
This map fails to be an algebra homomorphism but only by a scalar factor on each Peter-Weyl block. Recall the definition of the ribbon element $\rib$ for $\sU$ in \eqref{eq:rib}.
\begin{thm}\label{thm:Z(B)-basis}
  The map
    \begin{align}\label{eq:iso-Z}
      \ltil_{(1\ot v)\cdot\Kmat}:\cA^\inv_{\sigma}\rightarrow Z(\Bcs)
    \end{align}
    is a surjective algebra homomorphism. Moreover, the set
    \begin{align}\label{eq:basis}
      \{ \ltil_{(1\ot v)\cdot\scrK}( \tr_{V(\lambda)_+,\sigma,q})\,|\,\lambda\in P^+, 
      \sigma(\lambda)=\lambda\}
     \end{align}  
  is a basis of the center $Z(\Bcs)$.
\end{thm}
\begin{proof}
  Let $\lambda\in P^+$ with $\sigma(\lambda)=\lambda$. By Propositions \ref{prop:inZ} and \ref{prop:Z-construct} we have
  \begin{align*}
     \ltil_{(1\ot v)\cdot\Kmat}(\tr_{V(\lambda)_+,\sigma,q})=
     q^{(\lambda,\lambda+2\rho)}\ltil_\scrK(\tr_{V(\lambda)_+,\sigma,q})
     \in Z(\Bcs)\cap \ad(\Uc)(K_{-2\lambda})\setminus\{0\}.
  \end{align*}  
  Hence, by Proposition \ref{prop:dim-estimate} the elements in the set \eqref{eq:basis} span the center $Z(\Bcs)$. These elements are also linearly independent because they lie in distinct components of the locally finite part with respect to the direct sum decomposition \eqref{eq:Fl-PW}. This proves the second part of the theorem.

It remains to show that the map \eqref{eq:iso-Z} is a homomorphism of algebras.
Let $a,b\in \cA^{\inv}_{\sigma}$. Using Property (3) in Theorem \ref{thm:scrK} and the first equation in \eqref{eq:rib-conditions} one calculates
\begin{align*}
  \ltil_{(1\ot v)\cdot\Kmat}(ab)&= \langle (\id \ot \kow)\big((1\ot v) \Kmat\big),a\ot b\rangle_{3,2}\\
  &=\langle (1\ot v \ot v)\Rmat_{23}^{-1}\Kmat_{13}\Rmat_{23} \Kmat_{12},a\ot b\rangle_{3,2}
\end{align*}  
Using relation \eqref{eq:q-trace-prop} for the element $b\in \cA^{\inv}_{\sigma}$ one obtains
\begin{align*}
   \ltil_{(1\ot v)\cdot\Kmat}(ab) =&\langle\Rmat_{23}^{-1},a_{(1)} \ot b_{(2)}\rangle_{3,2} \langle \Rmat_{23},a_{(2)}\ot S^{-2}(b_{(1)}) \rangle_{3,2} \\
       &\quad\langle (1\ot \rib\ot \rib)\Kmat_{13}\Kmat_{12}, a_{(3)}\ot b_{(3)}\rangle_{3,2}\\
      =&\langle\Rmat_{23},S(a_{(1)}) \ot b_{(2)}\rangle_{3,2} \langle \Rmat^{-1}_{23},S(a_{(2)})\ot b_{(1)} \rangle_{3,2}\\ 
       &\quad\langle (1\ot \rib\ot \rib)\Kmat_{13}\Kmat_{12}, a_{(3)}\ot b_{(3)}\rangle_{3,2}\\
      =&   \langle (1\ot \rib\ot \rib)\Kmat_{13}\Kmat_{12}, a\ot b\rangle_{3,2}\\
      =&\ltil_{(1\ot v)\cdot\Kmat}(a)\,\ltil_{(1\ot v)\cdot\Kmat}(b).
\end{align*}
This shows that the map \eqref{eq:iso-Z} is an algebra homomorphism. 
\end{proof}
By Theorem \ref{thm:Z(B)-basis} the center $Z(\Bcs)$ has a distinguished basis \eqref{eq:basis}. Recall the notation
\begin{align*}
  P_{Z({\Bcs})}=\{\lambda\in P^+\,|\,\sigma(\lambda)=\lambda\}
\end{align*}
from Remark \ref{rem:PZB} and Lemma \ref{lem:PZB}. For simplicity we set $b_\lambda=\ltil_{(1\ot v)\cdot \Kmat}(\tr_{V(\lambda)_+,\sigma,q})$ for all $\lambda\in P_{Z({\Bcs})}$. Let $Z(\Bcs)_\Z$ denote the $\Z$-linear span of the basis $\{b_\lambda\,|\,\lambda\in P_{Z(\Bcs)}\}$. The ring $\cA^\inv_\sigma$ also has a distinguished basis consisting of the quantum traces of simple $\Uc_\sigma$-modules, see Lemma \ref{lem:Osigma-traces}. Let $\cA^\inv_{\sigma,\Z}$ denote the $\Z$-linear span of this basis. Let $K_0(\Oint^{\langle \sigma\rangle})$ denote the Grothendieck ring of the tensor category $\Oint^{\langle \sigma \rangle}$. The multiplicative behaviour of the quantum traces in $\cA_\sigma$ implies that the $\Z$-linear map
 \begin{align*}
   \Psi:K_0(\Oint^{\langle \sigma\rangle}) \rightarrow \cA^\inv_{\sigma,\Z}, \qquad [V]\mapsto \tr_{V,\sigma,q}
 \end{align*}  
 defines an isomorphism of $\Z$-algebras. Combining this with Theorem \ref{thm:Z(B)-basis} we obtain the following result.
 \begin{cor}\label{cor:ZBZ}
   There is a surjective $\Z$-algebra homomorphism
   \begin{align*}
     \Phi: K_0(\Oint^{\langle \sigma\rangle}) \rightarrow Z(\Bcs)_\Z
   \end{align*}
   with the following properties:
   \begin{enumerate}
   \item If $\sigma=\id$ then $\Phi([V(\lambda)])=b_\lambda$ for all $\lambda\in P^+$. In this case $\Phi$ is an isomorphism of $\Z$-algebras.
   \item If $\sigma \neq \id$ then
     \begin{align}\label{eq:VVW}
       \Phi([V(\lambda)_+])=b_\lambda=-\Phi([V(\lambda)_-]),\qquad
       \Phi([W(\mu)])=0
     \end{align}
     for all $\lambda,\mu\in P^+$ with $\sigma(\lambda)= \lambda$ and $\sigma(\mu)\neq \mu$. In this case
     \begin{align*}
       \ker(\Phi)= \mathrm{span}_\Z\{[V(\lambda)_+]-[V(\lambda)_-], [W(\mu)]\,|\,\lambda,\mu\in P^+, \sigma(\lambda)=\lambda, \sigma(\mu)\neq \mu\}.
     \end{align*}  
   \end{enumerate}  
 \end{cor}  
\begin{proof}
  As explained above the map $\Phi=\ltil_{(1\ot v)\Kmat}\circ \Psi:K_0(\Oint^{\langle \sigma\rangle}) \rightarrow Z(\Bcs)_\Z$ is a well defined surjective $\Z$-algebra homomorphism. If $\sigma=\id$ then $\Phi$ is an isomorphism by the second part of Theorem \ref{thm:Z(B)-basis}. If $\sigma\neq \id$ and $\lambda,\mu\in P^+$ with $\sigma(\lambda)=\lambda$ and $\sigma(\mu)\neq\mu$ then the definition of $\Kmat$ in \eqref{eq:KBU-def} leads to
\begin{align*}
  \ltil_{(1\ot v)\cdot\Kmat}(\tr_{V(\lambda)_-,\sigma,q})&= -\ltil_{(1\ot v)\cdot\Kmat}(\tr_{V(\lambda)_+,\sigma,q})\\
  \ltil_{(1\ot v)\cdot\Kmat}(\tr_{W(\mu),\sigma,q})&=0.
\end{align*}
This implies formulas \eqref{eq:VVW}. The description of $\ker(\Phi)$ again follows from the second part of Theorem \ref{thm:Z(B)-basis}.
\end{proof}
\begin{rema}
  By Theorem \ref{thm:Z(B)-basis} the center $Z(\Bcs)$ has a distinguished basis $\{b_\lambda\,|\,\lambda\in P_{Z(\Bcs)}\}$. Corollary \ref{cor:ZBZ} implies that there exist coefficients $\eta_{\lambda,\mu}^\nu\in \N_0$ for all $\lambda,\mu, \nu \in P_{Z(\Bcs)}$ such that
  \begin{align*}
    b_\lambda b_\mu= \sum_{\nu\in P_{Z(\Bcs)}} \eta_{\lambda,\mu}^\nu b_\nu.
  \end{align*} 
If $\sigma=\id$ then the $\eta_{\lambda,\mu}^\nu$ are the Littlewood-Richardson coefficients for $\gfrak$. By the second part of Corollary \ref{cor:ZBZ} the coefficients  $\eta_{\lambda,\mu}^\nu$ can also be determined in the case $\sigma\neq \id$. As an example consider $\Uc$ for $\gfrak=\slfrak_3(\C)$ and the quantum symmetric pair coideal subalgebra $\Bcs$ for $X=\emptyset$ and $\tau=\id$. Let $\varpi_1$, $\varpi_2$ denote the fundamental weights for $\gfrak$. In this case the center $Z(\Bcs)$ has a basis $\{b_{k(\varpi_1+\varpi_2)}\,|\,k\in \N_0\}$. Decomposing tensor products of $\Uc_\sigma$-modules we obtain
\begin{align*}
  V(\varpi_1+\varpi_2)_+\ot V(\varpi_1+\varpi_2)_+\cong \, & V(2\varpi_1+2\varpi_2)_+\oplus W(3\varpi_1)\\ &\oplus V(\varpi_1+\varpi_2)_+\oplus V(\varpi_1+\varpi_2)_- \oplus V(0)_+.
\end{align*}
Hence we get $b_{\varpi_1 + \varpi_2} b_{\varpi_1 + \varpi_2} = b_{2(\varpi_1 + \varpi_2)} +1$ in this case.
\end{rema}
\begin{rema}
  The results of Section \ref{sec:construction} build on the existence of the universal K-matrix $\Kmat$ which is only established for parameters $\bc=(c_i)_{i\in I\setminus X}\in \cC$, $\bs=(s_i)_{i\in I\setminus X}\in \cS$ satisfying the additional conditions \eqref{eq:c-bar-condition}, \eqref{eq:s-X-condition}, and \eqref{eq:c-tau-condition}. In \cite[Corollary 8.4]{a-KL08} the fact that
  \begin{align}\label{eq:Z-estimate}
    \dim\big(Z(\Bcs)\cap \ad(\Uq)(K_{-2\lambda})\big)\ge 1 \qquad \mbox{if $\lambda\in P_{Z(\Bcs)}$}
  \end{align}
  was established for all $\bc\in \cC$, $\bs\in \cS$ without any additional requirements on the parameters. Distinguishing three cases, we note here that the present section also provides a proof of \eqref{eq:Z-estimate} for more general parameter families $\bc\in \cC$, $\bs\in \cS$, albeit not quite in complete generality.
  \begin{enumerate}
    \item If \eqref{eq:Z-estimate} holds for $\bs={\mathbf 0}=(0,\dots,0)$ then \eqref{eq:Z-estimate} also holds for general $\bs\in \cS$. This follows from the facts that $B_{\bc,{\mathbf 0}}$ and $\Bcs$ are isomorphic as algebras \cite[Theorem 7.1]{a-Letzter03}, \cite[Theorem 7.1]{a-Kolb14}, and that the isomorphism preserves lowest weight components with respect to the left adjoint action of $\Uq$. 
    \item If $\tau(j)=j$ for some $j\in I\setminus X$ and $\lambda\in \field(q^{1/d})$ then there exists a Hopf algebra automorphism $\phi:\Uc\rightarrow \Uc$ such that $\phi(\Bcs)=B_{\bc',\bs}$ where $\bc'=(c_i')_{i\in I \setminus X}\in \cC$ is determined by 
    \begin{align*}
      c_i'=\begin{cases}
              \lambda^2 c_i & \mbox{if $i=j$},\\
              c_i & \mbox{else,}
           \end{cases}
    \end{align*} 
    see \cite[Section 9]{a-Kolb14}.
    This means that the methods of this section prove that \eqref{eq:Z-estimate} holds for arbitrary values of $c_j\in \field(q^{1/d})$ with $\bc\in \cC$ up to adjoining square roots to $\field(q^{1/d})$.
    \item If $\tau(j)\neq j$ for some $j\in I\setminus X$ then applying a Hopf algebra automorphism of $\Uc$ to $\Bcs$ only multiplies $c_j$ and $c_{\tau(j)}$ by the same factor or interchanges them. This means that if $(\alpha_j,\Theta(\alpha_j))\neq 0$ then the methods of this section do not prove \eqref{eq:Z-estimate} for general $c_j, c_{\tau(j)}$ with $\bc\in \cC$. 
  \end{enumerate}
\end{rema}
\subsection{A twisted quantum trace}\label{sec:tqt}
Assume that $\sigma \neq \id$. In this case $Z(\Bcs)_\Z$ is also isomorphic to a $\Z$-subalgebra of $\cA$ which can be explicitly described. Indeed, consider the twisted left adjoint action of $\Uc$ on itself defined by
\begin{align*}
  \ad_l^{\sigma}(u)(x)=u_{(1)}x S(\sigma_U(u_{(2)})) \qquad \mbox{for all $u,x\in \Uc$.}
\end{align*} 
Let $\ad_l^{\sigma,\ast}$ denote the induced right action of $\Uc$ on $\cA$. More explicitly, one has
$\big(\ad_l^{\sigma,\ast}(u)(a)\big)(x)=a\big(u_{(1)}x S(\sigma_U(u_{(2)}))\big)$ for all $u,x\in \Uc$, $a\in \cA$. Define 
\begin{align*} 
  \cA^{\inv,\sigma} = \{a\in \cA\,|\,\ad_l^{\sigma,\ast}(u)(a) = \vep(u)a \mbox{ for all $u\in \check{U}$}\}.
\end{align*}  
One checks that $ \cA^{\inv,\sigma}$ is a subalgebra of $\cA$. For any $a\in \cA_\sigma^\inv$ and $u\in \Uc$ the restriction $a(\,\cdot\,\sigma)|_\Uc$ satisfies the relation
\begin{align*}
  \ad_l^{\sigma,\ast}(u)\big(a(\,\cdot\,\sigma)|_\Uc\big)
  &=a(u_{(1)}\,\cdot\,S(\sigma_U(u_{(2)}))\sigma)|_\Uc\\
  &=a(u_{(1)}\,\cdot\,\sigma S(u_{(2)}))|_\Uc\\
  &=\vep(u)a(\,\cdot\,\sigma)|_\Uc.
\end{align*}  
Hence there is a well defined algebra homomorphism
\begin{align*}
  \pi_\sigma: \cA^\inv_\sigma \rightarrow \cA^{\inv,\sigma}, \qquad a\mapsto a(\,\cdot\,\sigma)|_\Uc.
\end{align*}
For any $\lambda,\mu\in P^+$ with $\sigma(\lambda)=\lambda$ and $\sigma(\mu)\neq \mu$ one has
\begin{align}\label{eq:pis}
  \pi_\sigma(\tr_{V(\lambda)_+,\sigma,q})=-\pi_\sigma(\tr_{V(\lambda)_-,\sigma,q}), \qquad
  \pi_\sigma(\tr_{W(\mu),\sigma,q})=0.
\end{align}
For any $\lambda\in P^+$ with $\sigma(\lambda)=\lambda$ define a twisted quantum trace $\tr^\sigma_{V(\lambda),q}\in \cA$ by
\begin{align*}
   \tr^\sigma_{V(\lambda),q}=\pi_\sigma(\tr_{V(\lambda)_+,\sigma,q}).
\end{align*}
By construction, $\tr^\sigma_{V(\lambda),q}$ is a non-zero element in the  Peter-Weyl block $C^\lambda$ corresponding to $\lambda$.
\begin{lem}\label{lem:tr-basis}
  The set $\{\tr_{V(\lambda),q}^\sigma\,|\,\lambda\in P^+, \, \sigma(\lambda)=\lambda\}$ is a basis of $\cA^{\inv,\sigma}$.
\end{lem}
\begin{proof}
  The Peter-Weyl block $C^\lambda \cong V(\lambda)^\ast \ot V(\lambda)$ is invariant under the right action 
  $\ad_l^{\sigma,\ast}$. It contains a one-dimensional trivial submodule if and only if $V(\lambda)\cong V(\sigma(\lambda))$ which is equivalent to $\sigma(\lambda)=\lambda$. By definition, the twisted quantum trace $\tr^\sigma_{V(\lambda),q}$ spans such a trivial submodule.
\end{proof}
Lemma \ref{lem:tr-basis}, the formulas in \eqref{eq:pis} and the second part of Corollary \ref{cor:ZBZ} now allow us to identify the center $Z(\Bcs)$ with the ring of twisted invariants $\cA^{\inv,\sigma}$.
\begin{cor}\label{cor:Asigma-iso}
  There is a uniquely determined algebra isomorphism $\varphi:\cA^{\inv,\sigma}\rightarrow Z(\Bcs)$ such that the following diagram commutes:
  $$
    \xymatrix{\cA^\inv_\sigma \ar[r]^{\ltil_{(1\ot v)\Kmat}} \ar[d]^{\pi_\sigma}& Z(\Bcs)\\
              \cA^{\inv,\sigma} \ar[ur]^\varphi& } 
    $$
    The isomorphism $\varphi$ is given on the basis from Lemma \ref{lem:tr-basis} by $\varphi(\tr^{\sigma}_{V(\lambda),q})= b_\lambda$
    for all $\lambda\in P^+$ with $\sigma(\lambda)=\lambda$.
\end{cor}  

\providecommand{\bysame}{\leavevmode\hbox to3em{\hrulefill}\thinspace}
\providecommand{\MR}{\relax\ifhmode\unskip\space\fi MR }
\providecommand{\MRhref}[2]{%
  \href{http://www.ams.org/mathscinet-getitem?mr=#1}{#2}
}
\providecommand{\href}[2]{#2}

\end{document}